\documentclass[12pt]{amsart}
\usepackage[foot]{amsaddr}
\usepackage[utf8x]{inputenc}
\usepackage{amsmath,amssymb,amsthm,amsfonts}
\usepackage{mathabx}
\usepackage[]{graphicx}
\usepackage{pgf,tikz}
\usepackage{pgfplots}
\usepackage{tikz-cd}
\usepackage{enumitem}
\usepackage{xcolor}
\usepackage{hyperref}
\usepackage{tabularx}
\usepackage[sc]{mathpazo}
\usepackage{eulervm}
\usepackage{comment}

\newtheorem{lemma}{Lemma}

\newtheorem{thm}{Theorem}
\newtheorem*{thm*}{Theorem}
\newtheorem{prop}{Proposition}
\newtheorem{example}{Example}

\newtheorem{rmk}{Remark}\theoremstyle{remark}
\newtheorem{defn}{\bf Definition}
\newtheorem*{defn*}{\bf Definition}

\usepackage{xspace}
\newcommand{\ol}[1]{\overline{#1}}

\newcommand{\etale}{\'{e}tale\xspace}

\usepackage{bbm}
\newcommand{\acts}{\; \rotatebox[origin=c]{-90}{$\circlearrowright$} \;}

\newcommand{\irr}{\mathrm{irr}}
\newcommand{\genirr}{\mathrm{gen.irr}}

\newcommand{\A}{\mathbb{A}}
\newcommand{\cO}{\mathcal{O}}
\newcommand{\CC}{\mathbf{C}}
\newcommand{\GL}{\mathrm{GL}}
\newcommand{\SL}{\mathrm{SL}}
\newcommand{\Z}{\mathbb{Z}}
\newcommand{\ZZ}{\mathbb{Z}}
\newcommand{\Q}{\mathbb{Q}}
\newcommand{\Qbar}{\ol{\Q}}

\DeclareMathOperator{\tr}{\mathrm{tr}}
\newcommand{\ot}{\otimes}
\newcommand{\sm}{\setminus}
\newcommand{\red}{{\tt red}}

\newcommand{\embed}{\hookrightarrow}

\newcommand{\onto}{\twoheadrightarrow}

\newcommand{\Spec}[1]{\mathrm{Spec}\left( #1 \right)}
\newcommand{\gpHom}[2]{\mathrm{Hom}_{\mathrm{gp}}\left( #1, #2 \right)}

\author{Benjamin Church\textsuperscript{1}}
\address{\textsuperscript{1} Stanford University}
\email{bvchurch@stanford.edu}
\thanks{The first author was partially supported by an NSF GRFP fellowship under grant DGE-2146755.}
\author{Francisco García-Cortés\textsuperscript{2}}
\address{\textsuperscript{2} Universidad de Sevilla}
\email{fragarcor3@alum.us.es}
\thanks{The second author was partially supported by grants PID2020-117843GB-I00 and PID2020-114613GB-I00 funded by MICIU/AEI/10.13039/501100011033.}

\title[Character Varieties and Weak Integrality]{$\mathrm{SL}_2$-character varieties of $2$-generated groups and failure of weak integrality}
\date{\today}

\begin{document}

\begin{abstract} Let $\ell$ be a prime number, $k$ a positive integer and consider the group $\Gamma_{\ell^k} :=\langle a,b\ \vert\ a^{\ell^k(\ell^k-1)}ba^{-\ell^k}b^{-2}\rangle$. We prove that $\Gamma_{\ell^k}$ is not $\SL_2$-weakly integral with obstruction at exactly the prime $\ell$. We also give a general description of the character varieties of $2$-generated groups with a relation of the form $a^{n_1} b^{m_1} a^{n_2} b^{m_2} = 1$.
\end{abstract}

\maketitle
\tableofcontents

\section{Introduction}

Recent work of Esnault, Esnault--Groechenig, and Esnault--de Jong \cite{Esnault23}, \cite{EG18}, \cite{EdJ24} has revealed obstructions to a finitely presented group $\pi$ being \textit{geometric} i.e. arising as the topological fundamental group of a smooth connected quasi-projective variety over $\CC$. Unlike previous Hodge-theoretic methods, these obstructions concern the arithmetic properties of the character variety of $\pi$ using deep results on the structure of arithmetic representations of the \etale fundamental groups of varieties in positive characteristic \cite{Dri12, Gai07}. Much of their work addresses the case of isolated points of the character variety i.e.\ those points corresponding to \textit{rigid} local systems. In a similar vein, \cite{EdJ24} defines a property of a group called \textit{weak integrality} that is global on the character variety:

\begin{defn}
Let $\pi$ be a finitely presented group. We say that $\pi$ is \textit{weakly integral} for a pair of positive integers $(r, \delta)$ \textit{at} a prime $\ell$ if there exists an absolutely irreducible representation
\[ \rho_\ell : \pi \to \GL_r(\ol{\Z}_\ell) \]
with $\det{\rho_{\ell}}$ having order dividing $\delta$ where $\ol{\Z}_{\ell}$ is the integral closure of $\Z_{\ell}$ inside $\Qbar_\ell$. Such a representation is necessarily valued in $\cO_{L}$ where $L / \Q_\ell$ is a finite extension. 
\end{defn}

\begin{defn}
We say that $\pi$ is \textit{weakly integral} if for any pair $(r, \delta)$ such that there exists an irreducible representation 
\[ \rho_{\CC} : \pi \to \GL_r(\CC) \]
with determinant of order dividing $\delta$ then $\pi$ is weakly integral for $(r, \delta)$ at $\ell$ for every prime $\ell$. 
\end{defn}

The main result of \cite{EdJ24} implies that any geometric $\pi$ is weakly integral. Furthermore, to demonstrate that this is indeed a new obstruction to a finitely presented group being geometric, they present an example due to  Becker--Breuillard--Varjú \cite{BreLetter,BBV22} of a $2$-generated group that is not weakly integral for $(r, \delta) = (2,1)$ at $\ell = 2$ but is weakly integral for all $\ell \neq 2$. This example is detailed in a letter of Breuillard kindly shared with us by Breuillard and Esnault \cite{BreLetter}; our calculations take inspiration from and often closely follow these ideas. Answering a question of Esnault \cite[Problem 3]{EsnPCMI}, we show that this example fits into a family of $2$-generated groups $\Gamma_{\ell^k}$ indexed by prime powers $\ell^k$ that are not weakly integral for $(2,1)$ exactly at the prime $\ell$. 

\begin{thm} \label{thm:gamma_ell}
Let $\Gamma_{\ell^k} = \big< a,b \mid a^{\ell^k ( \ell^k - 1)} b a^{-\ell^k} b^{-2} \big>$. Then the generically irreducible $\SL_2$-character variety $M^{\genirr}(\Gamma_{\ell^k}, \SL_2) \to \Spec{\Z}$ consists of exactly $2k$ ($k$ if $\ell = 2$) irreducible components each isomorphic to $\Spec{\Z[\zeta_{2\ell^{2k}}, \zeta_3, 1/\ell]}$ and hence admits $\ol{\Z}_{\ell'}$-points for all $\ell' \neq \ell$ and whose image is $\Spec{\Z[\ell^{-1}]} \subset \Spec{\Z}$. In particular, $\Gamma_{\ell^k}$ is weakly integral for $(2,1)$ at exactly all primes except $\ell$.
\end{thm}

Furthermore, we study the structure of generically irreducible\footnote{the components of the character variety whose generic point corresponds to an irreducible representation, see Definition~\ref{defn:irred_subschemes}} $\SL_2$-character varieties, denoted $M^{\genirr}(\pi, \SL_2)$, of $2$-generated groups presented by a relation of the form $a^{n_1} b^{m_1} a^{n_2} b^{m_2} = 1$. We characterize when they have components of specified dimension and when these components are integral at specified primes. This will follow from a decomposition of the character variety arising from ``universal relations'' developed in \S \ref{section:charvars} following from structure theory of the algebra $H[\pi]$ introduced in \S \ref{section:algebras}. Defining equations for $\SL_2$-character varieties were also studied in \cite{ABL17} and \cite{HP23} following a related strategy. For groups as above, our description reduces every equation cutting out the character variety to combinations of Chebyshev polynomials in such a way that their dependence on the exponents becomes comprehensible.

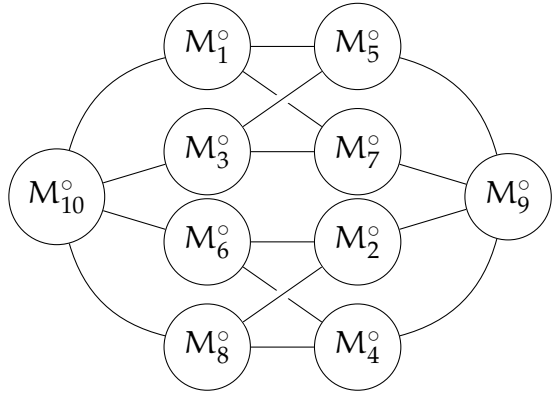
\begin{figure}[h!] \label{fig:intersection_components}
\begin{center}
\begin{tikzpicture}
\node[draw, circle, minimum size=15pt] (p10) at (0,0) {$M^\circ_{10}$};
\node[draw, circle, minimum size=15pt] (p1) at (2,2) {$M^\circ_1$};
\node[draw, circle, minimum size=15pt] (p3) at (2,.6) {$M^\circ_3$};
\node[draw, circle, minimum size=15pt] (p6) at (2,-.6) {$M^\circ_6$};
\node[draw, circle, minimum size=15pt] (p8) at (2,-2) {$M^\circ_8$};
\node[draw, circle, minimum size=15pt] (p5) at (4,2) {$M^\circ_5$};
\node[draw, circle, minimum size=15pt] (p7) at (4,.6) {$M^\circ_7$};
\node[draw, circle, minimum size=15pt] (p2) at (4,-.6) {$M^\circ_2$};
\node[draw, circle, minimum size=15pt] (p4) at (4,-2)  {$M^\circ_4$};
\node[draw, circle, minimum size=15pt] (p9) at (6,0) {$M^\circ_9$};
\draw (p10) to [bend left] (p1);
\draw (p10) to (p3);
\draw (p10) to (p6);
\draw (p10) to [bend right] (p8);
\draw (p1) to (p5);
\draw (p1) to (p7);
\draw [line width=5pt, white] (p3) to (p5);
\draw (p3) to (p5);
\draw (p3) to (p7);
\draw (p6) to (p2);
\draw (p6) to (p4);
\draw [line width=5pt, white] (p8) to (p2);
\draw (p8) to (p2);
\draw (p8) to (p4);
\draw (p5) to [bend left](p9);
\draw (p7) to (p9);
\draw (p2) to (p9);
\draw (p4) to [bend right] (p9);
\end{tikzpicture}
\end{center}
\caption{The maximum intersection graph (all intersection graphs are a subgraph of this one).}
\end{figure}

\begin{thm} \label{thm:components}
Let $\pi = \left< a,b \mid a^{n_1} b^{m_1} a^{n_2} b^{m_2} \right>$. There is a decompostion,
\[ M^{\genirr}(\pi, \SL_2) = \bigcup_{i = 1}^{10} M_i^{\circ} \]
into closed subschemes that are unions of ireducible components (some of which may be empty) with the following properties:
\begin{enumerate}
    \item each $M_i^\circ$ is reduced hence $M^{\genirr}(\pi, \SL_2)$ is reduced
    \item for $i = 1, \dots, 8$ suppose $M_i^{\circ} \neq \emptyset$ then $\dim{(M_i^{\circ})_{\Q}} \ge 1$ and $M_i^{\circ}(\Z[\mu_{\infty}]) \neq \emptyset$
    \item for $i = 9$ suppose $n_1 \neq n_2$ and $\dim{(M_i^{\circ})_{\Q}} > 0$ then $M_i^{\circ}(\ZZ[\mu_\infty]) \neq \emptyset$
    \item for $i = 10$ suppose $m_1 \neq m_2$ and $\dim{(M_i^{\circ})_{\Q}} > 0$ then $M_i^{\circ}(\ZZ[\mu_\infty]) \neq \emptyset$
\end{enumerate}
Hence non-integrality is only possible when all $M_i^\circ = \emptyset$ for $i = 1, \dots, 8$ and $\dim{(M_i^\circ)_{\Q}} = 0$ for $i = 9,10$. Moreover, the intersection graph between the $M_i^\circ$ is a subgraph of Figure~\ref{fig:intersection_components}. 
Furthermore, the dimensions of the $M_i^\circ$ are computed in \S \ref{section:dimensions}.
\end{thm}

Keeping in mind Esnault and de Jong's theorem, it is particularly interesting to look for components of $M^{\genirr}(\pi, \SL_2)$ that do not dominate $\Spec{\Z}$ and moreover intersect the generic fiber inside $M^{\irr}(\pi, \SL_2)_{\Q}$. Our next result fully describes when such components appear. 

\begin{thm} \label{thm:integrality}
If $m_1 + m_2 = \pm 1$ or $n_1 + n_2 = \pm 1$ then $(M_i^\circ)_{\Q} = \emptyset$ for $i = 1, \dots, 8$ and $\dim{(M_i^\circ)_{\Q}} \le 0$ for $i = 9,10$. Let $s = n_1 - n_2$ and $t = m_1 - m_2$ and $\zeta_{2s}, \zeta_t \in \ol{\Q} \sm \{ \pm 1 \}$ be not necessarily primitive $2s$ and $t$-th roots of unity respectively. If moreover
\[\begin{aligned}z = &\frac{\zeta_{2s}^{n_1+1} - \zeta_{2s}^{-(n_1+1)}}{\zeta_{2s}^{n_1} - \zeta_{2s}^{-n_1}} \cdot \frac{\zeta_t^{m_1 - 1} - \zeta_t^{-(m_1 - 1)}}{\zeta_t^{m_1} - \zeta_t^{-m_1}}\\ &\quad\quad+ \frac{\zeta_t^{m_1 + 1} -\zeta_t^{-(m_1 + 1)}}{\zeta_t^{m_1} - \zeta_t^{-m_1}} \cdot \frac{\zeta_{2s}^{n_1-1} - \zeta_{2s}^{-(n_1 - 1)}}{\zeta_{2s}^{n_1} - \zeta_{2s}^{-n_1}}\end{aligned} \]
is not an algebraic integer then there is a closed immersion 
\[ \Spec{\Z[\zeta_{2s}, \zeta_t, z]} \embed M_{9}^\circ \]
such that the fiber over $\Q$ lies in $M^{\irr}(\pi, \SL_2)_{\Q}$. Furthermore $(M_9^{\circ})_{\Q}$ is covered by the union of these inclusions over all choices of $\zeta_{2s}$ and $\zeta_t$.
Hence $M_9^\circ$ has a $\ol{\Z}_{\ell}$-point over exactly those $\ell$ for which $v_{\ell}(z) \ge 0$ for some $\ell$-adic valuation and some $z$. Swapping the role of $n, m$, the same statement holds for $M_{10}^\circ$. Hence $\pi$ is $(2,1)$ weakly integral at exactly this set of $\ell$. 
\end{thm}

\begin{rmk}
We will find that $M^{\genirr}(\Gamma_\ell, \SL_2)$ behaves rather differently for $\ell = 2$ and $\ell \neq 2$ with respect to this decomposition. For $\ell = 2$, it turns out that $M^{\genirr}(\Gamma_\ell, \SL_2) = M_9^\circ$ is irreducible (over $\Q$) and $M_{10}^\circ = \emptyset$. However, for $\ell \neq 2$ there is a decomposition $M^{\genirr}(\Gamma_\ell, \SL_2) = M_9^\circ \cup M_{10}^\circ$ both are irreducible (over $\Q$) of dimension $0$ and abstractly isomorphic. 
\end{rmk}

The proofs of these results are contained in \S\ref{section:dimensions}. Theorem~\ref{thm:gamma_ell} follows from our decomposition results (e.g.\ Theorem~\ref{thm:integrality}) but we give a self-contained proof in \S \ref{section:gamma_ell_reps} using only the structure theory of $H[\Gamma_{\ell^k}]$ as a replacement for the conjugacy arguments in \cite{BreLetter} that rely on the relation taking a particular form.

\subsection*{Acknowledgements} 

We thank H\'{e}l\`{e}ne Esnault for asking the question that lead to this work as well as invaluable advice and encouragement. We would also like to thank Emmanuel Breuillard for answering our questions and his helpful correspondence. The authors would like to thank the IAS Park City Mathematics Institute for organizing the summer 2024 session on motivic homotopy theory during which this project was completed. 

\subsection{Notations and conventions on character varieties}

Let $\pi$ be a finitely presented group and $G$ a reductive group over a ring $R$. The $G$-\textit{character variety} $M(\pi, G) \to \Spec{R}$ is the GIT quotient
\[ M(\pi, G) := \gpHom{\pi}{G} // G \]
where $G \acts \gpHom{\pi}{G}$ by conjugation and $\gpHom{\pi}{G}$ is the affine $A$-scheme representing the functor
\[ A / R \mapsto \{ \text{ group homomorphisms } \rho : \pi \to G(A) \} \]
Over a field $k$, we say that $\rho : \pi \to G(k)$ is \textit{irreducible} if $\rho(\pi)$ is not contained in any proper parabolic subgroup of $G$. Additionally, we say $\rho$ is \textit{completely reducible} or \textit{semi-simple} if for every parabolic $P \subset G$ containing $\rho(T)$, there is a Levi subgroup $L \subset P$ contaning $\rho(\pi)$. It is well-known (e.g. \cite[Thm.~30]{Sik09}) that $\rho \in \gpHom{\pi}{G}$ is poly-stable for the conjugation action if and only if $\rho$ is semi-simple and if $\rho$ is moreover irreducible then it is stable for the conjugation action.  
\par 
Note that we are interested in $G = \SL_2$ for which a representation $\rho : \pi \to \SL_2(k)$ is irreducible if and only if $\rho(\pi)$ is not contained in a Borel subgroup if and only if the corresponding representation $\rho : \pi \to \GL_2(k)$ is irreducible in the usual sense.

\begin{defn} \label{defn:irred_subschemes}
We consider two subschemes 
\[ M^{\irr}(\pi, G) \subset M^{\genirr}(\pi, G) \subset M(\pi, G) \] 
of the $G$-character variety defined over $\Spec{\Z}$, where $M^{\irr}(\pi, G)$ is the subscheme consisting of (everywhere/strongly) irreducible points and $M^{\genirr}(\pi, G)$ consists of those components whose generic point corresponds to an irreducible representation \textit{over} $\CC$. Explicitly, $M^{\irr}(\pi, G)$ is the GIT quotient of the subscheme of $\gpHom{\pi}{G}$ representing the subfunctor of those representations that are irreducible at each geometric point. Likewise, define $M^{\genirr}(\pi, G)$ as the scheme-theoretic image of 
\[ M^{\irr}(\pi, G)_{\Q} \to M(\pi, G). \]
\end{defn}

For an overview on $\mathrm{SL}_2(\CC)$-character varieties of finitely presented groups, the reader can consult \cite{ABL17}.

\subsection{$\SL_2$-character varieties of $2$-generated groups} \label{section:algebras}

The study of representations of $2$-generated groups begins with understanding the character variety of $F_2$, the free group on two generators. A classic result states:

\begin{thm}[Fricke--Vogt] \cite[Proposition 3.2]{brumfiel}
$M(F_2, \SL_2) \cong \A^3$ where the isomorphism is given by
\[ \Z[x,y,z] \to R = \Z[A_{ij}, B_{ij}]/(\det{A} - 1, \det{B} - 1) \]
sending $x \mapsto \tr{A}$ and $y \mapsto \tr{B}$ and $z \mapsto \tr{AB}$ where $R$ is the coordinate ring of $\gpHom{F_2}{\SL_2}$.
\end{thm} 

This result is exactly the following algebraic statement: the ring of invariants $R^{\SL_2}$ under conjugation is generated by $\tr{A}, \tr{B}, \tr{AB}$ and these are algebraically independent.
\par 
Given a $2$-generated group $\pi$ we fix a set of two generators i.e. a surjection $F_2 \onto \pi$. This presents $M(\pi, \SL_2) \subset M(F_2, \SL_2) \cong \A^3_{x,y,z}$ as a subscheme. We will now understand the equations cutting out $M^{\genirr}(\pi, \SL_2)$ inside $\A^3_{x,y,z}$. 
\par 
Following \cite{brumfiel} we define an algebra $H[\pi]$ that encodes the Cayley--Hamilton relations universal to any $\SL_2$-representation:

\begin{defn}
Let $\pi$ be a finitely presented group. Let 
\[ H[\pi] := \ZZ[\pi]/(h(g + g^{-1}) - (g + g^{-1})h)_{g,h \in \pi} \]
\end{defn}

The relations of this algebra arise from the observation that if $\rho : \pi \to \SL_2(R)$ is a representation then by Cayley--Hamilton
\[ \rho(g)^{-1} = \tr{\rho(g)} \cdot I - \rho(g) \]
and therefore $\rho(g) + \rho(g)^{-1}$ is a central matrix so its commutator with any element lies in the kernel of $\tilde{\rho} : \ZZ[\pi] \to M_2(R)$. Due to this interpretation, we define a $\ZZ$-linear map $T : H[\pi] \to H[\pi]$ called the ``trace'' defined as $T(x) = x + \iota(x)$ where $\iota : H[\pi] \to H[\pi]$ is the $\ZZ$-linear extension of $g \mapsto g^{-1}$. The trace $T(g)$ is mapped to $\tr{\rho(g)} \cdot I$ under any $\ZZ$-algebra map $\tilde{\rho} : H[\pi] \to M_2(A)$ extending a representation $\rho : \pi \to \SL_2(R)$. Hence $\SL_2(R)$-representations of $\pi$ correspond to involution preserving algebra maps $H[\pi] \to M_2(R)$ (where $M_2(R)$ is given the involution $A \mapsto \mathrm{adj}(A)$, defined by $A^{-1} \cdot \det{A}$ on $A$ invertible). Let $TH[\pi]$ be the subalgebra of $H[\pi]$ generated by the trace elements $T(x)$ for $x \in H[\pi]$. This is a commutative subalgebra.
\par 
For $2$-generated groups, there is the following structure result for $H[\pi]$. Let $F_2 \onto \pi$ be a presentation and $a,b \in \pi$ the images of the generators. Let $x = T(a), y = T(b), z = T(ab)$ and consider the element
\[ \red = x^2 + y^2 + z^2 - xyz - 4 = 
T(aba^{-1}b^{-1}) - 2 \in TH[\pi] \]
In a representation $\rho$, the vanishing of $\tilde{\rho}(\red)$ detects commutativity of $\rho(a)$ and $\rho(b)$. Hence $\red$ detects absolute irreducibility by the following proposition:

\begin{prop} \cite[Proposition 4.1]{brumfiel} \label{prop:red_irreducible}
The following are equivalent:
\begin{enumerate}
    \item $1 \ot \tilde{\rho} : R \ot_{TH[\pi]} H[\pi] \to M_2(R)$ is an isomorphism
    \item $\rho : \pi \to \SL_2(R)$ is absolutely irreducible e.g. is irreducible for each geometric point $R \to \bar{k}$
    \item $\tilde{\rho}(\red) \in R$ is a unit.
\end{enumerate}
\end{prop}

Therefore, we get a description of the absolutely irreducible character variety:

\begin{prop} 
Under the isomorphism $M(F_2, \SL_2) \xrightarrow{\sim} \A^3_{x,y,z}$ given by $(x,y,z) \mapsto (\tr{A}, \tr{B}, \tr{AB})$ we have $M^{\irr}(F_2, \SL_2) \xrightarrow{\sim} D(\red)$, the basic where $\red$ is nonvanishing. Therefore, $M^{\genirr}(F_2, \SL_2) = M(F_2, \SL_2)$.
\end{prop}

Furthermore, $\red$ is central to the following structure result for the algebra $H[\pi]$ over $TH[\pi]$. 

\begin{prop} \cite[Proposition 3.5]{brumfiel} \label{prop:basis}
If $\phi : TH[\pi] \to R$ is a ring homomorphism such that $\phi(\red)$ is not a zero-divisor then $R \otimes_{TH[\pi]} H[\pi]$ is free of rank $4$ over $R$ with basis 
\[ \{ 1 \ot 1, 1 \ot a, 1 \ot b, 1 \ot ab \} \]
\end{prop}

The following result allows us to reconstruct a representation, up to conjugaction, from its traces.

\begin{prop} \cite[Proposition 4.3]{brumfiel}
Let $\phi : TH[\pi] \to R$ be a homomorphism such that $\phi(\red) \in R$ is not a zero divisor. Then $\phi$ extends to an involution preserving algebra homomorphism $\tilde{\rho} : H[\pi] \to M_2(R')$ where $R'$ is either $R$ or a quadratic extension of $R$. 
\end{prop}

Suppose $\pi = \left< a,b \mid w(a,b) \right>$ for some word $w(a,b)$. Since in $TH[F_2] = \Z[x,y,z]$ the element $\red$ is not a zero divisor, we see that in $H[\pi]$ we can write $w(a,b)$ in terms of the basis $1, a, b, ab$,
\[ w(a,b) = c_1 1 + c_a a + c_b b + c_{ab} ab \]
for $c_1, c_a, c_b, c_{ab} \in TH[F_2] = \Z[x,y,z]$. Thus we get explicit generators for the ideal $I$ whose vanishing locus, $V(I)$, defines the generically irreducible character varieties.

\begin{prop} \label{prop:explicit_equations}
The locally closed subscheme $M^{\irr}(\pi, \SL_2) \subset M(F_2, \SL_2)$ is identified with the open set $D(\red)$ of $V(c_1 - 1, c_a, c_b, c_{ab}) \subset \A^3$. Under this identification, $M^{\genirr}(\pi, \SL_2)$ is cut out by the $w$-elimination ideal $I \subset \Z[x,y,z]$ of $(c_1 - 1, c_a, c_b, c_{ab}, w \cdot \red - 1) \subset \Q[x,y,z,w]$.
\end{prop}

\begin{proof}
The natural map on trace-like elements $\phi|_{TH} : TH[F_2] \to TH[\pi]$ is exactly the pullback map of the embedding of affine schemes $M(\pi, \SL_2) \subset M(F_2, \SL_2)$. By definition, $\phi|_{TH}$ is surjective because $\phi$ is surjective, $\phi(T(x)) = T(\phi(x))$, and $TH[\pi]$ is generated by elements of the form $T(x)$. 
\par 
By definition, the natural quotient map $\phi : H[F_2] \to H[\pi]$ satisfies $\phi(w(a,b)) = 1$. Because $H[F_2]$ is free over $TH[F_2]$ with basis $1,a,b,ab$ we can write
\[ \phi(w(a,b)) = \phi(c_1) 1 + \phi(c_a) a + \phi(c_b) b + \phi(c_{ab}) ab = 1. \]
We would like to compare the coefficients to obtain the kernel of $\phi|_{TH}$. However, $1, a, b, ab$ may not form a basis of $H[\pi]$ over $TH[\pi]$ if $\red$ is a zero-divisor but by Proposition~\ref{prop:basis} these elements are independent in $H[\pi][\red^{-1}]$ over $TH[\red^{-1}]$. Hence, under the natural map $\phi_{\red} : H[F_2] \to H[\pi][\red^{-1}]$ we obtain $\phi_{\red}(c_1) = 1$ and $\phi_{\red}(c_a) = \phi_{\red}(c_b) = \phi_{\red}(c_{ab}) = 0$. If $\rho : H[F_2] \to R$ satisfies $\rho(c_1 - 1, c_a, c_b, c_{ab}) = 0$ then clearly $\rho(w(a,b)) = \rho(1)$ so $\rho$ factors through $H[F_2] \to H[\pi]$. Hence these elements, along with those annihilated by $\red$, generate the kernel of $TH[F_2] \to TH[\pi][\red^{-1}]$ which corresponds to the immersion $M^{\irr}(\pi, \SL_2) \subset M(F_2, \SL_2)$. 
\par 
By definition, $M^{\genirr}(\pi, \SL_2)$ is the scheme-theoretic image of $M^{\irr}(\pi, \SL_2)_{\Q} \to M(\pi, \SL_2)$ or equivalently in $M(F_2, \SL_2)$ since the former is closed in the latter. This elimination ideal $I$ defines exactly the scheme-theoretic image in $M(F_2, \SL_2)$ so we are done.
\end{proof}

\subsection{Some facts about $\mathrm{SL}_2$ matrices} 

\subsubsection{Definition of $c_n(t),d_n(t)\in\Z[t]$} 
After Cayley--Hamilton, every $M \in \mathrm{SL}_2(\mathbb{C})$ satisfies the equation $M^2=\mathrm{tr}(M)\cdot M - 1$. We define the polynomials $c_n(t)$, $d_n(t)$ by imposing that the formula 
\[ M^n = c_n(\mathrm{tr}(M))\cdot M + d_n(\mathrm{tr}(M))\] 
holds for all $M$. They are monic and satisfy the recursive formulas 
\[ c_{n+1}(t)=tc_n(t)+d_n(t), \quad d_{n+1}(t)=-c_{n}(t). \] 
Moreover, since $M^{-1}=-M+\mathrm{tr}(M)$ and $\mathrm{tr}(M^{-1})=\mathrm{tr}(M)$, we find that \[\begin{aligned}M^{-n} &= (M^{-1})^n = c_{n}(\mathrm{tr}(M^{-1}))\cdot M^{-1} + d_{n}(\mathrm{tr}(M^{-1}))\\ &= -c_n(\mathrm{tr}(M))\cdot M + \mathrm{tr}(M)c_n(\mathrm{tr}(M))+d_n(\mathrm{tr}(M)),\end{aligned}\] hence $c_{-n}(t)=-c_n(t)$ and $d_{-n}(t)=tc_n(t)+d_n(t)=c_{n+1}(t)$.

Define $c_0(t):=0, d_0(t):=1$. Then $c_n(t), d_n(t)\in\Z[t]$ for every $n\in\Z$ and they are either zero or monic. Moreover, for $|n|>1$ we have $\mathrm{deg}_t~c_n=|n|-1$ and $\mathrm{deg}_t~d_n=|n|-2$.

\subsubsection{Recursive relations in matrix form} Writing the recursive relationships of the polynomials $c_n(t), d_n(t)$ in matrix form we see that 
\[\begin{pmatrix}c_{n+1}(t)\\ d_{n+1}(t)\end{pmatrix} = \begin{pmatrix}t &1\\ -1 &0\end{pmatrix}\begin{pmatrix}c_n(t)\\ d_n(t)\end{pmatrix}.\] 
Since $c_0(t)=0, d_0(t)=1$, if we denote $R(t) :=\left(\begin{matrix}t &1\\ -1 &0\end{matrix}\right) \in \SL_2(\Z[t])$,
\[\begin{pmatrix}c_n(t)\\ d_n(t)\end{pmatrix} = R(t)^n \begin{pmatrix}0\\ 1\end{pmatrix}.\] 
Observe that the characteristic polynomial of $R$ is $\lambda^2-t\lambda+1$ and also \[R(t)^n=\begin{pmatrix}c_{n+1}(t) &c_n(t)\\ d_{n+1}(t) &d_n(t)\end{pmatrix} = \begin{pmatrix}tc_n(t)+d_n(t) & c_n(t)\\ -c_n(t) &d_n(t) \end{pmatrix}.\] 
Hence
\begin{equation} \label{eq:determinant_condition}
1=\det(R^n)=c_n(t)^2+tc_n(t)d_n(t)+d_n(t)^2.
\end{equation} 
In particular, $c_n(t)$ and $d_n(t)$ do not vanish simultaneously.
\par 
Also, it is easy to see that 
\[\begin{pmatrix}\pm2 & 1\\ -1 &0\end{pmatrix}^n = \begin{pmatrix}(\pm1)^nn+1 &(\pm1)^{n-1}n\\ -(\pm1)^{n-1}n &-(\pm1)^{n}(n-1)\end{pmatrix},\] 
so $c_n(\pm2)=(\pm1)^{n-1}n$ and $d_n(\pm2)=-(\pm1)^{n}(n-1)$.

\subsubsection{Factorization and divisibility properties} \label{section:polynomials}

Let $t_0\in \ol{\Q}$ be the trace of a matrix $M \in \SL_2(\ol{\Q})$ such that $c_n(t_0)=0$. Then $M^n=\pm I$, which is possible if and only if $M^{2n}=I$. Therefore, $M$ is a diagonalizable matrix and its eigenvalues are $2n$-th roots of unity. Therefore $t_0 = \zeta_{2n}^k+\zeta_{2n}^{-k} = 2\cos(k\pi/n)$ for some $k=1,\dots,n-1$ and all of this values are possible. Since $c_n(t)$ is a monic polynomial of degree $|n|-1$ (for $n\neq0$) and $2\cos(k\pi/n)$ is a root for every $k=1,\dots,n-1$ it follows that 
\[c_n(t)=\prod_{k=1}^{n-1}\left(t-2\cos\left(\frac{k\pi}{n}\right)\right).\] 
In particular, every root of $c_n(t)$ is the trace of some matrix in $\mathrm{SL}_2(\mathbb{Q}(\mu_\infty))$. Then 
\begin{align*}
\begin{pmatrix} c_{n+m}(t) \\ d_{n+m}(t)\end{pmatrix} &= R^{n+m} \begin{pmatrix}0\\ 1\end{pmatrix} = R^n \begin{pmatrix}c_m(t) \\ d_m(t) \end{pmatrix}
\\
&= \begin{pmatrix}c_{n+1}(t) &c_n(t)\\ d_{n+1}(t) &d_n(t)\end{pmatrix} \begin{pmatrix}c_m(t) \\ d_m(t) \end{pmatrix} 
\\
& = \begin{pmatrix} c_{n+1}(t) c_m(t) + c_n(t) d_m(t) \\ d_{n+1}(t) c_m(t) + d_n(t) d_m(t) \end{pmatrix}\\ &=\begin{pmatrix}tc_n(t)c_m(t)+c_n(t)d_m(t)+d_n(t)c_m(t)\\ d_n(t)d_m(t)-c_n(t)c_m(t)\end{pmatrix}
\end{align*}
giving a product formula.

By definition, $M^{nm}=(M^n)^m = c_m(\mathrm{tr}(M^n))M^n+d_m(\mathrm{tr}(M^n)$. Since $M^n=c_n(t)M+d_n(t)$ and $\mathrm{tr}(M^n)=tc_n(t)+2d_n(t)$, we get \[\begin{aligned}c_{nm}(t)&=c_m(tc_n(t)+2d_n(t))\cdot c_n(t),\\ d_{nm}(t)&=c_m(tc_n(t)+2_n(t))\cdot d_n(t)+d_m(tc_n(t)+2d_n(t)).\end{aligned}\] In particular, $c_n(t)~|~c_{nm}(t)$.

Moreover, the following will be useful.
\begin{lemma}\label{lemma:gcd}
$c_{\mathrm{gcd}(n,m)}(t)=\mathrm{gcd}(c_n(t),c_m(t))$.
\end{lemma}

\begin{proof}
Denote $g=\mathrm{gcd}(n,m)$. We already know $c_g(t)~|~c_n(t),c_m(t)$; so $(c_n(t),c_m(t))\subset (c_g(t))\subset\Z[t]$. But also, if we write Bezout's identity $g=an+bm$, then \[\begin{aligned}c_g(t) &=c_{an+bm}(t)\\ &= d_{bm}(t)c_{an}(t)+c_{am+1}(t)c_{bm}(t)\\ &= d_{bm}(t)c_a(tc_n(t)+2d_n(t))\cdot c_n(t)\\ &\quad\quad+c_{an+1}(t)c_b(tc_m(t)+2d_m(t))\cdot c_m(t).\end{aligned}\] This implies $(c_g(t))\subset (c_n(t),c_m(t))$ and we get the equality of the ideals.
\end{proof}

\subsubsection{Eigenvalues and $c_n(t)$} If we assume $M$ diagonalizable we can express $c_n(\mathrm{tr}(M))$ using the eigenvalues. Indeed, let $\mu_1\neq \mu_2$ be the eigenvalues of $M\neq \pm I$, then \[\begin{aligned}\begin{pmatrix} \mu_1 &0\\ 0 &\mu_2\end{pmatrix}^n &= c_n(\mu_1+\mu_2)\cdot \begin{pmatrix}\mu_1 &0\\ 0 &\mu_2\end{pmatrix} + d_n(\mu_1+\mu_2) \\ &= c_n(\mathrm{tr}(M))\cdot \begin{pmatrix}\mu_1 &0\\ 0 &\mu_2\end{pmatrix} + d_n(\mathrm{tr}(M)).\end{aligned}\] Therefore $\mu_i^n = \mu_i c_n(\mathrm{tr}(M)) + d_n(\mathrm{tr}(M))$ for $i=1,2$ and \[c_n(\mathrm{tr}(M)) = \frac{\mu_1^n-\mu_2^n}{\mu_1-\mu_2}.\]

\subsubsection{Relation to Chebyshev polynomials}

Write $T_m$ (resp.\ $U_m$) to denote the Chebyshev polynomials of the first (resp. second) kind \cite{Chebyshev}. Then, using the definition, it is easy to see that \[tc_n(t)+2d_n(t)=2T_n(t/2),\quad\text{ and }\quad c_n(t)=U_{n-1}(t/2).\] In fact, this description would give all the identities stated above. We prove them here because our arguments are self contained (e.g.\ compare with \cite{Chebyshev}).

\section{Irreducible $\SL_2$-representations of $\Gamma_{\ell^k}$} \label{section:gamma_ell_reps}

For any prime $\ell$ and any positive integer $k$ we consider the finitely presented group $\Gamma_{\ell^k} := \langle a, b\ |\ a^{\ell^k(\ell^k-1)}ba^{-\ell^k}b^{-2}\rangle$.

\subsection{The $\SL_2$-character variety of $\Gamma_{\ell^k}$}

Recall that the canonical embedding $M^\genirr(\Gamma_{\ell^k},\mathrm{SL}_2) \subset \A^3_{x,y,z}$ says that an irreducible representation is determined uniquely up to conjugacy by the three traces $x,y,z$.

For $\ell^k = 2$, the following appeared in \cite{BreLetter, BBV22}. We give a uniform argument using notation developed in the previous section.

\begin{prop}
Let $k$ be a positive integer. For $\ell=2$ the scheme $M^\genirr(\Gamma_{2^k},\mathrm{SL}_2)\subset \A^3_{x,y,z}$ has exactly $k$ irreducible components (irreducible over $\Spec{\Z}$ i.e.\ $\Q$-Galois orbits of the geometric generic fiber). The first $k-1$ components are \[(M_{10,i}^\circ):\left\{\begin{aligned}&\mathrm{ord}\ \rho(a)=2^i,\\ &y-1=0,\\ &zc_{2^k}(x)-yc_{2^k+1}(x)=0,\end{aligned}\right.\] for $i=k+2,\dots,2k$, and the last one is \[(M_{9}^\circ):\left\{\begin{aligned}&\mathrm{ord}\ \rho(a)=2^{2k+1},\\ &y+1=0,\\ &zc_{2^k}(x)-yc_{2^k+1}(x)=0.\end{aligned}\right.\]
For odd $\ell$ the scheme $M^\genirr(\Gamma_{\ell^k},\mathrm{SL}_2)\subset \A^3_{x,y,z}$ has exactly $2k$ irreducible components. They  are \[ (M_{10,i}^\circ):\left\{\begin{aligned}&\mathrm{ord}\ \rho(a)=\ell^i,\\ &y-1=0,\\ &zc_{\ell^k}(x)-yc_{\ell^k+1}(x)=0,\end{aligned}\right.\text{ and }(M_{9,i}^\circ):\left\{\begin{aligned}&\mathrm{ord}\ \rho(a)=2\ell^i,\\ &y+1=0,\\ &zc_{\ell^k}(x)-yc_{\ell^k+1}(x)=0,\end{aligned}\right.\] for $i=k+1,\dots,2k$.
\end{prop}

\begin{rmk}
If $\ell^k=2$ there are no $\Q$-Galois orbits $M_{10,i}^\circ$ since $k+2=3>2=2k$.
\end{rmk}

\begin{proof}
Let $N$ be the scheme explicitly described in the theorem statement. We will show that $M^\irr(\Gamma_{\ell^k},\mathrm{SL}_2)_\Q\subset N$ with equality over $\Q$.

Let $K$ be a field of characteristic zero. A group homomorphism $\rho:\Gamma_{\ell^k}\to\mathrm{SL}_2(K)$ is described by two matrices $a\mapsto A,b\mapsto B$ such that $A^{\ell^k(\ell^k-1)}BA^{-\ell^k}B^{-2}=I$. This is equivalent to the condition $A^{\ell^k(\ell^k-1)}B=B^2A^{\ell^k}$. Assuming $\rho$ is absolutely irreducible we can write both sides in terms of $I,A,B$ and $AB$ using Cayley--Hamilton and compare component wise: \[\begin{aligned}c_{\ell^k(\ell^k-1)}(x)\cdot AB+d_{\ell^k(\ell^k-1)}(x)\cdot B = &-yc_{\ell^k}(x)\cdot AB\\ &+y(xc_{\ell^k}(x)+d_{\ell^k}(x))\cdot B\\ &+(y^2-1)c_{\ell^k}(x)\cdot A\\ &+(y(z-xy)c_{\ell^k}(x)-d_{\ell^k}(x))\cdot I. \end{aligned}\] Therefore, \[\begin{aligned}y(z-xy)c_{\ell^k}(x)-d_{\ell^k}(x)&=0,\\ (y^2-1)c_{\ell^k}(x)&=0,\\ yc_{\ell^k+1}(x)&=d_{\ell^k(\ell^k-1)}(x),\\ -yc_{\ell^k}(x)&= c_{\ell^k(\ell^k-1)}(x).\end{aligned}\] Since $c_{\ell^k}(x)$ and $d_{\ell^k}(x)$ do not vanish simultaneously, we find using the first two equations that $c_{\ell^k}(x)\neq0, y^2-1=0$ and $zc_{\ell^k}(x)-yc_{\ell^k+1}(x)=0$. The last two equations imply the following matrix equations and the reverse implication holds as long as $\rho$ is irreducible \[\begin{aligned}A^{\ell^k(\ell^k-1)}&=c_{\ell^k(\ell^k-1)}(x)\cdot A + d_{\ell^k(\ell^k-1)}(x)\cdot I\\ &= y\cdot(-c_{\ell^k}(x)\cdot A+c_{\ell^k+1}(x)\cdot I)\\ &=y\cdot(c_{-\ell^k}(x)\cdot A+d_{-\ell^k}(x)\cdot I)\\ &= y\cdot A^{-\ell^k},\end{aligned}\] that is, $A^{\ell^{2k}}=y\cdot I$.

Because we may assume $\rho$ absolutely irreducible, $A\neq\pm I$ and since $c_{\ell^k}(x)\neq 0$ this implies $A^{\ell^k}\neq\pm I$. In particular, $A^{\ell^i}\neq\pm I$ for any $i=0,\dots,k$. Since the only $\mathrm{SL}_2$ matrices of order $2$ are $\pm I$, we deduce that $A^{2\ell^i}\neq I$ for any $i\in\{0,\dots,k\}$. Consider the following two cases:

\begin{enumerate}
\item If $\ell=2$ we see that $A^{2^{k+1}}\neq I$. Since $A^{2^{2k}}=y\cdot I$, the order of $A$ is a $2$-power $2^i$ with $i\in\{k+2,\dots,2k+1\}$. Observe that if $y-1=0$ then $i\leq 2k$. Also note that if $y+1=0$ then $A^{2^{2k}}=-I$ and this implies that the order of $A$ can not be $\leq2^{2k}$ since $2$ is even, so $\mathrm{ord}\ A=2^{2k+1}$.
\item For odd $\ell$, if $y-1=0$ then $A^{\ell^{2k}}=I$ so $\mathrm{ord}\ A$ is an $\ell$-power $\ell^i$ with $i\in\{k+1,\dots,2k\}$. If $y+1=0$ then $A^{\ell^{2k}}=-I$, in particular $A^{\ell^i}\neq I$ for every $i\in\{0,\dots,2k\}$. Hence the order of $A$ is not an $\ell$-power, so it must be $2$ times an $\ell$-power.
\end{enumerate}

This discussion shows that $M^\irr(\Gamma_{\ell^k},\mathrm{SL}_2)_\Q\subset\A^3$ is contained in the scheme cut out by the above equations. To show equality on the fiber over $\Q$, we must demonstrate that $\red$ is nonzero at each point. Since Proposition~\ref{prop:integrality_failure} only relies on the above inclusion, we defer a complete argument until we have performed those calculations. In the proof we show that for any $(x,y,z)$ satisfying the polynomial conditions in the hypothesis, $x,y$ are $\ell$-adic integral and $z$ has negative $\ell$-adic valuation. Therefore $\red(x,y,z)=x^2+y^2+z^2-xyz-4$ is nonzero, otherwise $z$ would be $\ell$-adic integral. Therefore, $(x,y,z)\in M^\irr(\Gamma_{\ell^k},\mathrm{SL}_2)$. Moreover, the equations \[c_{2\ell^{2k}}(x)=0, y^2-1=0, zc_{\ell^k}(x)-yc_{\ell^k+1}(x)=0, c_{\ell^k}(x)\neq 0\] have finite support (since $c_{\ell^k}(x)\neq0$ fixes the value of $z$ given the values $x,y$ which themselves are fixed by the first two equations). Hence $M^\irr(\Gamma_{\ell^k},\mathrm{SL}_2)_\Q\subset N$ with equality after taking the fiber over $\Q$. Since $N_\Q\embed N\subset \A^3_\Z$ presents $N$ as the scheme theoretic image of $N_\Q\to\A^3_\Z$ we conclude that $N=M^\genirr(\Gamma_{\ell^k},\mathrm{SL}_2)$ as subschemes of $\A^3_\Z$.
\end{proof}

\subsection{Failure of weak integrality for $\Gamma_{\ell^k}$}

Write $\zeta_r=\exp(2\pi i/r)$ for the standard $r$-th root of unity. Let $\hat\nu_\ell$ be the unique $\ell$-adic valuation of $\Q(\zeta_{2\ell^{2k}})$ extending the $\ell$-adic valuation $\nu_\ell$ of $\Q$. Since it is unique, it is given by the formula \[\hat\nu_{\ell,k}(x) = \frac{1}{[\Q(\zeta_{2\ell^{2k}}):\Q]}\cdot\nu_\ell\left(\mathrm{Nm}_{\Q(\zeta_{2\ell^{2k}})/\Q}(x)\right).\] If $\ell$ is odd then $\Q(\zeta_{2\ell^{2k}})=\Q(\zeta_{\ell^{2k}})$ and $[\Q(\zeta_{2\ell^{2k}}):\Q]=\ell^{2k-1}(\ell-1)$. If $\ell=2$ then $[\Q(\zeta_{2^{2k+1}}):\Q]=2^{2k}$.

\begin{prop}\label{prop:integrality_failure}
$\Gamma_{\ell^k}$ is weakly integral for $(2,1)$ (meaning with respect to $\mathrm{SL}_2$-valued representations) at $\ell'$ if and only if $\ell'\neq\ell$.
\end{prop}

\begin{proof}
Let $(x,y,z)$ be a $K$-valued point of $M^\genirr(\Gamma_{\ell^k},\mathrm{SL}_2)$ for some field $K$ of characteristic zero. 
\begin{enumerate}
\item Let $\ell=2$. The order of $\rho(a)$ being always finite implies $\rho(a)$ is diagonalizable with eigenvalues $\zeta_{2^i}^m, \zeta_{2^i}^{-m}$ for some odd $m$ and $i\in \{k+2,\dots,2k+1\}$. Hence $x=\zeta_{2^i}^m+\zeta_{2^i}^{-m}$, which is a real algebraic integer.

Furthermore $z=yc_{\ell^k+1}(x)/c_{\ell^k}(x)$ with $y^2-1=0$, so $x,y,z\in \Q(\zeta_{2^{2k+1}})$ and $(x,y,z)$ is actually a $\Q(\zeta_{2^{2k+1}})$-valued point. Writing $z$ using the eigenvalues of $\rho(a)$ we find \[z=y\frac{\zeta_{2^i}^{m(2^k+1)}-\zeta_{2^i}^{-m(2^k+1)}}{\zeta_{2^i}^{2^km}-\zeta_{2^i}^{-2^km}}=\frac{y}{\zeta_{2^i}^{m}}\frac{1-\zeta_{2^{i-1}}^{m(2^k+1)}}{1-\zeta_{2^{i-k-1}}^{m}}.\] Therefore \[\begin{aligned}\hat\nu_2(z) &= \frac{1}{2^{2k}}\nu_2\left(\frac{\mathrm{Nm}_{\Q(\zeta_{2^{i-1}})/\Q}\big(1-\zeta_{2^{i-1}}^{m(2^k+1)}\big)^{[\Q(\zeta_{2^{2k+1}}):\Q(\zeta_{2^{i-1}})]}}{\mathrm{Nm}_{\Q(\zeta_{2^{i-k-1}})/\Q}\big(1-\zeta_{2^{i-k-1}}^m\big)^{[\Q(\zeta_{2^{2k+1}}):\Q(\zeta_{2^{i-k-1}})]}}\right)\\ &= \frac{\nu_2(2^{2^{2k+2-i}})-\nu_2(2^{2^{3k+2-i}})}{2^{2k}} = \frac{1-2^k}{2^{i-2}}\in(-1,0).\end{aligned}\]

\item Let $\ell\neq2$. Since $\rho(a)$ has finite order, it is diagonalizable. We split the argument according to wether $y-1=0$ or $y+1=0$.

If $y-1=0$, then the eigenvalues of $\rho(a)$ are $\zeta_{\ell^i}^m,\zeta_{\ell^i}^{-m}$ for some $m$ prime to $\ell$ and $i\in\{k+1,\dots,2k\}$. Then $x=\zeta_{\ell^i}^m+\zeta_{\ell^i}^{-m}$, which is a real algebraic integer. Write $z$ using the eigenvalues of $\rho(a)$ to obtain the expression \[z=\frac{1}{\zeta_{\ell^i}^{m}}\frac{1-\zeta_{\ell^i}^{2m(\ell^k+1)}}{1-\zeta_{\ell^{i-k}}^{2m}}.\]

If $y+1=0$, the eigenvalues of $\rho(a)$ are $\zeta_{2\ell^i}^m,\zeta_{2\ell^i}^{-m}$ for some $m$ prime to $2\ell$ and $i\in\{k+1,\dots,2k\}$. Then $x=\zeta_{2\ell^i}^m+\zeta_{2\ell^i}^{-m}$, which is a real algebraic integer. As before, we can express $z$ by \[z=\frac{-1}{\zeta_{2\ell^i}^m}\frac{1-\zeta_{\ell^i}^{m(\ell^k+1)}}{1-\zeta_{\ell^{i-k}}^{m}}.\]

In both cases, $x,y,z\in\Q(\zeta_{2\ell^{2k}})$ so $(x,y,z)$ is a $\Q(\zeta_{2\ell^{2k}})$-valued point. Also, observe that for a fixed $i$, the $\ell$-adic valuation of $z$ is the same independently of the value of $y$. Without loss of generality assume $y-1=0$. It follows that \[\begin{aligned}\hat\nu_\ell(z)&= \frac{1}{\ell^{2k-1}(\ell-1)}\nu_\ell\left(\frac{\mathrm{Nm}_{\Q(\zeta_{\ell^{i}})/\Q}\big(1-\zeta_{\ell^i}^{2m(\ell^k+1)}\big)^{[\Q(\zeta_{\ell^{2k}}):\Q(\zeta_{\ell^i})]}}{\mathrm{Nm}_{\Q(\zeta_{\ell^{i-k}})/\Q}\big(1-\zeta_{\ell^{i-k}}^{2m}\big)^{[\Q(\zeta_{\ell^{2k}}):\Q(\zeta_{\ell^{i-k}})]}}\right)\\ &= \frac{\nu_\ell(\ell^{\ell^{2k-i}})-\nu_\ell(\ell^{\ell^{3k-i}})}{\ell^{2k-1}(\ell-1)}=\frac{1-\ell^k}{\ell^{i-1}(\ell-1)}\in(-1,0).\end{aligned}\]
\end{enumerate}

Now for $\ell'\neq\ell$ we show that $\Gamma_{\ell^k}$ is weakly integral. Both $x$ and $y$, the traces of $\rho(a)$ and $\rho(b)$, are integral and in fact lie in $\Z[\zeta_{2\ell^{2k}}]$. It suffices to consider $z=\tr(\rho(ab))$. Up to units, we can write \[z=\frac{1-\zeta_{\ell^i}^{m(\ell^k+1)}}{1-\zeta_{\ell^{i-k}}^{m}}\] for some $i>k$ and $m$ prime to $\ell$. For $m$ prime to $\ell$, \[\mathrm{Nm}_{\Q(\zeta_{\ell^i})/\Q}(1-\zeta_{\ell^i}^m)=\prod_{(j,\ell)=1} (1-\zeta_{\ell^i}^{jm}) = \Phi_{\ell^i}(1) = \ell\in\Z,\] and therefore $1-\zeta_{\ell^i}$ is a unit in any $\ell'$-adic local field (the sum of the valuations over $\nu_{\ell'}$ is zero and each is nonnegative). In fact, the same argument shows that $z$ is an $\ell'$-adic unit.
\end{proof}

\section{Character varieties of certain $2$-generated groups} \label{section:charvars}

In this section, we give a general computation of the character variety of a group of the form
\[ \pi = \langle a,b\ \vert\ a^{n_1}b^{m_1}a^{n_2}b^{m_2}\rangle. \]

Recall that ${\tt red} = x^2+y^2+z^2-xyz-4$ detects absolute irreducibility in $M(F_2, \SL_2) = \A^3_{x,y,z}$. The strategy is to find universal relations arising from writing 
\[ 1 = w(a,b) = a^{n_1}b^{m_1}a^{n_2}b^{m_2} = c_1 1 + c_a a + c_b b + c_{ab} ab \]
in terms of the basis $1, a, b, ab$ of $H[F_2]$ over $TH[F_2] = \Z[x,y,z]$ and obtain equations by comparing the expansions componentwise.

Although this procedure may produce traces of reducible representations, all absolutely irreducible representations of $\pi$ are captured in this way (by Proposition~\ref{prop:basis}). Using the polynomial relations
\[ M^n = c_n(\mathrm{tr}(M))\cdot M + d_n(\mathrm{tr}(M)) \]
we expand the word formally in terms of $x,y,z$ and the expressions $c_{n_1}(x), d_{n_1}(x), c_{n_2}(x), d_{n_2}(x), c_{m_1}(y), d_{m_1}(y), c_{m_2}(y), d_{m_2}(y)$.
For actual integer values of $n_1, m_1, n_2, m_2$ there are of course relations between these polynomials in $x$ and $y$. However, there are also universal relations arising only from the structure of the expansion. Choose new variables $\gamma_{n_i}, \gamma_{m_i}, \delta_{n_i}, \delta_{m_i}$ representing the formal value of $c_{n_i}, c_{m_i}, d_{n_i}, d_{m_i}$.
In the polynomial ring 
\[R:=\Z[x,y,z][\gamma_{n_1},\delta_{n_1},\gamma_{m_1},\delta_{m_1},\gamma_{n_2},\delta_{n_2},\gamma_{m_2},\delta_{m_2}] \] 
we can form the ideal
\[ I := (c_1 - 1, c_a, c_b, c_{ab}) \]
in $R$ and consider the structure of the scheme $\Spec{R/I}$.
\par
Because it produces shorter and more convenient expressions, we will actually rewrite the relation as 
\[ a^{n_2} b^{m_2} = b^{-m_1} a^{-n_1}. \]
Repeating the previous procedure, we obtain an equally valid but simplier form for the universal relations.  Expanding these expressions 
\begin{align*}
     b^{-m_1}a^{-n_1} &= r_{1} 1 + r_a a + r_b b + r_{ab} ab
     \\
     a^{n_2}b^{m_2} &= \ell_1 1+\ell_a a + \ell_b b + \ell_{ab} ab
\end{align*}
in terms of the basis $1$, $a$, $b$, $ab$ of $H[F_2]$ over $TH[F_2]=\Z[x,y,z]$ and identifying the coefficients we obtain the following equations: 
\[\left\{\begin{aligned}d_{n_2}(x)d_{m_2}(y)&=d_{n_1}(x)d_{m_1}(y)+xc_{n_1}(x)d_{m_1}(y)\\ &\quad\quad+yd_{n_1}(x)c_{m_1}(y)+zc_{n_1}(x)c_{m_1}(y),\\ c_{n_2}(x)d_{m_2}(y)&=-c_{n_1}(x)d_{m_1}(y),\\ d_{n_2}(x)c_{m_2}(y)&= -d_{n_1}(x)c_{m_1}(y),\\ c_{n_2}(x)c_{m_2}(y)&= -c_{n_1}(x)c_{m_1}(y).\end{aligned}\right.\]
As before, we substitute new abstract variables $\delta, \gamma$ and use these equations to from an ideal. However, there is another source of universal relations: the determinantal relations arising from equation~(\ref{eq:determinant_condition}). Including these defines an ideal
\[
 I :=(\ell_1 - r_1, \ell_a- r_a, \ell_b- r_b, \ell_{ab} - r_{ab}, \text{det's}). \]
Rearranging, we obtain the following presentation for $I$ which will be useful later:
 \begin{align*}I:=(&x\gamma_{n_1}\delta_{m_1}+y \delta_{n_1} \gamma_{m_1} + z\gamma_{n_1}\gamma_{m_1} + \delta_{n_1}\delta_{m_1}-\delta_{n_2}\delta_{m_2},
 \\ &\ \gamma_{n_1}\delta_{m_1}+\gamma_{n_2}\delta_{m_2}, 
 \\ &\ \delta_{n_1}\gamma_{m_1}+\delta_{n_2}\gamma_{m_2},
 \\ &\ \gamma_{n_1}\gamma_{m_1}+\gamma_{n_2}\gamma_{m_2},
 \\ &\ \gamma_{n_i}^2+x\gamma_{n_i}\delta_{n_i}+\delta_{n_i}^2-1\text{ for }i=1,2,
 \\ &\ \gamma_{m_i}^2+y\gamma_{m_i}\delta_{m_i}+\delta_{m_i}^2-1\text{ for }i=1,2).
 \end{align*}

The next lemma illustrates the sort of universal relations we are interested in:

\begin{lemma}\label{lemma:univ_rel} Let $I\subset R$ be the ideal introduced before. Then 
\[ \gamma_{n_1}^2-\gamma_{n_2}^2,\gamma_{m_1}^2-\gamma_{m_2}^2\in I. \]
\end{lemma}
\begin{proof}
Observe that \[\begin{aligned}\gamma_{n_1}^2-\gamma_{n_2}^2=&(\gamma_{n_1}\delta_{m_1}-\gamma_{n_2}\delta_{m_2}-y\gamma_{n_2}\delta_{m_2})(\gamma_{n_1}\delta_{m_1}+\gamma_{n_2}\delta_{m_2})\\ &\quad+(\gamma_{n_1}\gamma_{m_1}-\gamma_{n_2}\gamma_{m_2}+y\gamma_{n_1}\delta_{m_1})(\gamma_{n_1}\gamma_{m_1}+\gamma_{n_2}\gamma_{m_2})\\ &\quad-\gamma_{n_1}^2(\gamma_{m_1}^2+y\gamma_{m_1}\delta_{m_1}+\delta_{m_1}^2-1)\\ &\quad+\gamma_{n_2}^2(\gamma_{m_2}^2+y\gamma_{m_2}\delta_{m_2}+\delta_{m_2}^2-1).\end{aligned}\] Therefore \[\begin{aligned}\gamma_{n_1}^2-\gamma_{n_2}^2\in (&\gamma_{n_1}\delta_{m_1}+\gamma_{n_2}\delta_{m_2},\gamma_{n_1}\gamma_{m_1}+\gamma_{n_2}\gamma_{m_2},\\ &\gamma_{m_i}^2+y\gamma_{m_i}\delta_{m_i}+\delta_{m_i}^2-1\text{ for }i=1,2)\subset I.\end{aligned}\] The proof for $\gamma_{m_1}^2-\gamma_{m_2}^2$ is the same but exchanging $n$ by $m$ and $y$ by $x$ in the previous expressions.
\end{proof}

\begin{rmk}
Recall that, plugging in the trace expressions for the $\gamma, \delta$ variables, the traces of any irreducible representation of $\pi$ will give a point of $R/I$. However, we did use irreducibility in a crucial way because $I$ is defined by matching coefficients which requires freeness. For irreducible representations this is justified by Proposition~\ref{prop:basis} and Proposition~\ref{prop:red_irreducible}.
\par 
However, the universal relations may hold for some reducible points but it is not guaranteed. Therefore, it might seem natural to work in the ring $R[w]$ and include the generator $w \cdot \red - 1$ in the definition of $I$. However, it is easily checked that the elimination ideal
\[ (I + (w \cdot \red - 1)) \cap R = I  \]
equals $I$. Therefore, we do not obtain any additional universal relations by passing to $\red^{-1}$ first cannot later be obtained by shrinking to the open $M^{\irr} \subset M^{\genirr}$ defined as $D(\red)$. Note, this does not mean that the elimination ideal defining $M^{\irr}(\pi, \SL_2)$ in Proposition~\ref{prop:explicit_equations} is trivially given by simply removing $w \cdot \red - 1$. Indeed, it might be larger. It only means that, independent of the word $w$, there are no universal extra elements we can add by virtue of incorporating irreducibility. 
\end{rmk}

Using Lemma \ref{lemma:univ_rel} it is straightforward (but tedious) to obtain the primary decomposition of $I$. Recall that $R/I$ is reduced so this is nothing more than the irreducible components of $V(I)\subset \A^{11}$.

\subsection{Components}

The zero scheme of $I$ has 10 components which we analyze by computing the primary decomposition of $I$. In fact, $\Spec{R/I}$ is reduced hence these primary ideals are prime. The following table records the irreducible components of $V(I)$ defined by their prime ideals. This computation was performed by the computational algebra system M\textsc{AGMA} V2.28-3 \cite{MAGMA}.

\begin{center}
\begin{tabular}{p{.45\textwidth}p{.45\textwidth}}
Component $\wp_1:$ &Component $\wp_2:$\\
$\begin{aligned}(&\gamma_{n_1},\ \gamma_{n_2},\ \delta_{n_1}+1,\ \delta_{n_2}+1,\\ &\gamma_{m_1}+\gamma_{m_2},\ y\gamma_{m_1}+\delta_{m_1}-\delta_{m_2},\\ &\gamma_{m_1}^2+y\gamma_{m_1}\delta_{m_1}+\delta_{m_1}^2-1)\end{aligned}$ &$\begin{aligned}(&\gamma_{n_1},\ \gamma_{n_2},\ \delta_{n_1}+1,\ \delta_{n_2}-1,\\ &\gamma_{m_1}-\gamma_{m_2},\ y\gamma_{m_1}+\delta_{m_1}+\delta_{m_2},\\ &\gamma_{m_1}^2+y\gamma_{m_1}\delta_{m_1}+\delta_{m_1}^2-1)\end{aligned}$\\
\end{tabular}
\begin{tabular}{p{.45\textwidth}p{.45\textwidth}}
Component $\wp_3:$ &Component $\wp_4:$\\
$\begin{aligned}(&\gamma_{n_1},\ \gamma_{n_2},\ \delta_{n_1}-1,\ \delta_{n_2}-1,\\ &\gamma_{m_1}+\gamma_{m_2},\ y\gamma_{m_1}+\delta_{m_1}-\delta_{m_2},\\ &\gamma_{m_1}^2+y\gamma_{m_1}\delta_{m_1}+\delta_{m_1}^2-1)\end{aligned}$ &$\begin{aligned}(&\gamma_{n_1},\ \gamma_{n_2},\ \delta_{n_1}-1,\ \delta_{n_2}+1,\\ &\gamma_{m_1}-\gamma_{m_2},\ y\gamma_{m_1}+\delta_{m_1}+\delta_{m_2},\\ &\gamma_{m_1}^2+y\gamma_{m_1}\delta_{m_1}+\delta_{m_1}^2-1,)\end{aligned}$\\
\end{tabular}
\begin{tabular}{p{.45\textwidth}p{.45\textwidth}}
Component $\wp_5:$ &Component $\wp_6$:\\
$\begin{aligned}(&\gamma_{m_1},\ \gamma_{m_2},\ \delta_{m_1}+1,\ \delta_{m_2}+1,\\ &\gamma_{n_1}+\gamma_{n_2},\ x\gamma_{n_1}+\delta_{n_1}-\delta_{n_2},\\ &\gamma_{n_1}^2+x\gamma_{n_1}\delta_{n_1}+\delta_{n_1}^2-1,)\end{aligned}$ &$\begin{aligned}(&\gamma_{m_1},\ \gamma_{m_2},\ \delta_{m_1}+1,\ \delta_{m_2}-1,\\ &\gamma_{n_1}-\gamma_{n_2},\ x\gamma_{n_1}+\delta_{n_1}+\delta_{n_2},\\ &\gamma_{n_1}^2+x\gamma_{n_1}\delta_{n_1}+\delta_{n_1}^2-1,)\end{aligned}$\\
\end{tabular}
\begin{tabular}{p{.45\textwidth}p{.45\textwidth}}
Component $\wp_7:$ &Component $\wp_8:$\\
$\begin{aligned}(&\gamma_{m_1},\ \gamma_{m_2},\ \delta_{m_1}-1,\ \delta_{m_2}-1,\\ &\gamma_{n_1}+\gamma_{n_2},\ x\gamma_{n_1}+\delta_{n_1}-\delta_{n_2},\\ &\gamma_{n_1}^2+x\gamma_{n_1}\delta_{n_1}+\delta_{n_1}^2-1)\end{aligned}$ &$\begin{aligned}(&\gamma_{m_1},\ \gamma_{m_2},\ \delta_{m_1}-1,\ \delta_{m_2}+1,\\ &\gamma_{n_1}-\gamma_{n_2},\ x\gamma_{n_1}+\delta_{n_1}+\delta_{n_2},\\ &\gamma_{n_1}^2+x\gamma_{n_1}\delta_{n_1}+\delta_{n_1}^2-1)\end{aligned}$\\
\end{tabular}
\begin{tabular}{p{.45\textwidth}p{.45\textwidth}}
Component $\wp_9:$ &Component $\wp_{10}:$\\
$\begin{aligned}(&\gamma_{n_1}+\gamma_{n_2},\ \delta_{n_1}+\delta_{n_2},\\ &\gamma_{m_1}-\gamma_{m_2},\ \delta_{m_1}-\delta_{m_2},\\ &z\gamma_{n_1}\gamma_{m_1}+x\gamma_{n_1}\delta_{m_1}\\ &\quad+y\delta_{n_1}\gamma_{m_1}+2\delta_{n_1}\delta_{m_1},\\ &\gamma_{n_1}^2+x\gamma_{n_1}\delta_{n_1}+\delta_{n_1}^2-1,\\ &\gamma_{m_1}^2+y\gamma_{m_1}\delta_{m_1}+\delta_{m_1}^2-1)\end{aligned}$ &$\begin{aligned}(&\gamma_{n_1}-\gamma_{n_2},\ \delta_{n_1}-\delta_{n_2},\\ &\gamma_{m_1}+\gamma_{m_2},\ \delta_{m_1}+\delta_{m_2},\\ &z\gamma_{n_1}\gamma_{m_1}+x\gamma_{n_1}\delta_{m_1}\\ &\quad+y\delta_{n_1}\gamma_{m_1}+2\delta_{n_1}\delta_{m_1},\\ &\gamma_{n_1}^2+x\gamma_{n_1}\delta_{n_1}+\delta_{n_1}^2-1,\\ &\gamma_{m_1}^2+y\gamma_{m_1}\delta_{m_1}+\delta_{m_1}^2-1)\end{aligned}$
\end{tabular}
\end{center}

\bigskip

\begin{rmk} \label{rmk:constraining_z}
Observe that for $\wp_1,\dots,\wp_8$, no generator involves the variable $z$. For $\wp_9, \wp_{10}$ notice that $z$ appears in exactly one equation and it appears linearly with a coefficient of $\gamma_{n_1} \gamma_{m_1}$. Also note that $\gamma_{n_1}\gamma_{m_1}\in \wp_{1}\cap\cdots\cap\wp_{8}$, which implies that $z$ will be uniquely determined if and only if we work in any of the constructible sets $V(\wp_9)\setminus V(\prod_{i\neq9}\wp_i)$ or $V(\wp_{10})\setminus V(\prod_{i\neq10}\wp_i)$.
\end{rmk}

We describe the non trivial intersections between components:
\bigskip\\

\begin{tabular}{p{.45\textwidth}p{.45\textwidth}}
Intersection $\wp_1+\wp_5:$ &Intersection $\wp_1+\wp_7:$\\
$\begin{aligned}(&\gamma_{n_1},\ \gamma_{n_2},\ \delta_{n_1}+1,\ \delta_{n_2}+1,\\ &\gamma_{m_1},\ \gamma_{m_2},\ \delta_{m_1}+1,\ \delta_{m_2}+1)\end{aligned}$ &$\begin{aligned}(&\gamma_{n_1},\ \gamma_{n_2},\ \delta_{n_1}+1,\ \delta_{n_2}+1,\\ &\gamma_{m_1},\ \gamma_{m_2},\ \delta_{m_1}-1,\ \delta_{m_2}-1)\end{aligned}$
\end{tabular}

\begin{tabular}{p{.45\textwidth}p{.45\textwidth}}
Intersection $\wp_2+\wp_6:$ &Intersection $\wp_2+\wp_8:$\\
$\begin{aligned}(&\gamma_{n_1},\ \gamma_{n_2},\ \delta_{n_1}+1,\ \delta_{n_2}-1,\\ &\gamma_{m_1},\ \gamma_{m_2},\ \delta_{m_1}+1,\ \delta_{m_2}-1)\end{aligned}$ &$\begin{aligned}(&\gamma_{n_1},\ \gamma_{n_2},\ \delta_{n_1}+1,\ \delta_{n_2}-1,\\ &\gamma_{m_1},\ \gamma_{m_2},\ \delta_{m_1}-1,\ \delta_{m_2}+1)\end{aligned}$
\end{tabular}

\begin{tabular}{p{.45\textwidth}p{.45\textwidth}}
Intersection $\wp_3+\wp_5:$ &Intersection $\wp_3+\wp_7:$\\
$\begin{aligned}(&\gamma_{n_1},\ \gamma_{n_2},\ \delta_{n_1}-1,\ \delta_{n_2}-1,\\ &\gamma_{m_1},\ \gamma_{m_2},\ \delta_{m_1}+1,\ \delta_{m_2}+1)\end{aligned}$ &$\begin{aligned}(&\gamma_{n_1},\ \gamma_{n_2},\ \delta_{n_1}-1,\ \delta_{n_2}-1,\\ &\gamma_{m_1},\ \gamma_{m_2},\ \delta_{m_1}-1,\ \delta_{m_2}-1)\end{aligned}$
\end{tabular}

\begin{tabular}{p{.45\textwidth}p{.45\textwidth}}
Intersection $\wp_4+\wp_6:$ &Intersection $\wp_4+\wp_8:$\\
$\begin{aligned}(&\gamma_{n_1},\ \gamma_{n_2},\ \delta_{n_1}-1,\ \delta_{n_2}+1,\\ &\gamma_{m_1},\ \gamma_{m_2},\ \delta_{m_1}+1,\ \delta_{m_2}-1)\end{aligned}$ &$\begin{aligned}(&\gamma_{n_1},\ \gamma_{n_2},\ \delta_{n_1}-1,\ \delta_{n_2}+1,\\ &\gamma_{m_1},\ \gamma_{m_2},\ \delta_{m_1}-1,\ \delta_{m_2}+1)\end{aligned}$
\end{tabular}

\begin{tabular}{p{.45\textwidth}p{.45\textwidth}}
Intersection $\wp_2+\wp_9:$ &Intersection $\wp_4+\wp_9:$\\
$\begin{aligned}(&\gamma_{n_1},\ \gamma_{n_2},\ \delta_{n_1}+1,\ \delta_{n_2}-1,\\ &\gamma_{m_1}-\gamma_{m_2},\ \delta_{m_1}-\delta_{m_2},\\ &y\gamma_{m_1}+2\delta_{m_1},\\ &\gamma_{m_1}^2+y\gamma_{m_1}\delta_{m_1}+\delta_{m_1}^2-1)\end{aligned}$ &$\begin{aligned}(&\gamma_{n_1},\ \gamma_{n_2},\ \delta_{n_1}-1,\ \delta_{n_2}+1,\\ &\gamma_{m_1}-\gamma_{m_2},\ \delta_{m_1}-\delta_{m_2},\\ &y\gamma_{m_1}+2\delta_{m_1},\\ &\gamma_{m_1}^2+y\gamma_{m_1}\delta_{m_1}+\delta_{m_1}^2-1)\end{aligned}$
\end{tabular}

\begin{tabular}{p{.45\textwidth}p{.45\textwidth}}
Intersection $\wp_5+\wp_9:$ &Intersection $\wp_7+\wp_9:$\\
$\begin{aligned}(&\gamma_{n_1}+\gamma_{n_2},\ \delta_{n_1}+\delta_{n_2},\\ &\gamma_{m_1},\ \gamma_{m_2},\ \delta_{m_1}+1,\ \delta_{m_2}+1,\\ &x\gamma_{n_1}+2\delta_{n_1},\\ &\gamma_{n_1}^2+x\gamma_{n_1}\delta_{n_1}+\delta_{n_1}^2-1)\end{aligned}$ &$\begin{aligned}(&\gamma_{n_1}+\gamma_{n_2},\ \delta_{n_1}+\delta_{n_2},\\ &\gamma_{m_1},\ \gamma_{m_2},\ \delta_{m_1}-1,\ \delta_{m_2}-1,\\ &x\gamma_{n_1}+2\delta_{n_1},\\ &\gamma_{n_1}^2+x\gamma_{n_1}\delta_{n_1}+\delta_{n_1}^2-1)\end{aligned}$
\end{tabular}

\begin{tabular}{p{.45\textwidth}p{.45\textwidth}}
Intersection $\wp_1+\wp_{10}:$ &Intersection $\wp_3+\wp_{10}:$\\
$\begin{aligned}(&\gamma_{n_1},\ \gamma_{n_2},\ \delta_{n_1}+1,\ \delta_{n_2}+1,\\ &\gamma_{m_1}+\gamma_{m_2},\ \delta_{m_1}+\delta_{m_2},\\ &y\gamma_{m_1}+2\delta_{m_1},\\ &\gamma_{m_1}^2+y\gamma_{m_1}\delta_{m_1}+\delta_{m_1}^2-1)\end{aligned}$ &$\begin{aligned}(&\gamma_{n_1},\ \gamma_{n_2},\ \delta_{n_1}-1,\ \delta_{n_2}-1,\\ &\gamma_{m_1}+\gamma_{m_2},\ \delta_{m_1}+\delta_{m_2},\\ &y\gamma_{m_1}+2\delta_{m_1},\\ &\gamma_{m_1}^2+y\gamma_{m_1}\delta_{m_1}+\delta_{m_1}^2-1)\end{aligned}$
\end{tabular}

\begin{tabular}{p{.45\textwidth}p{.45\textwidth}}
Intersection $\wp_6+\wp_{10}:$ &Intersection $\wp_8+\wp_{10}:$\\
$\begin{aligned}(&\gamma_{n_1}-\gamma_{n_2},\ \delta_{n_1}-\delta_{n_2},\\ &\gamma_{m_1},\ \gamma_{m_2},\ \delta_{m_1}+1,\ \delta_{m_2}-1,\\ &x\gamma_{n_1}+2\delta_{n_1},\\ &\gamma_{n_1}^2+x\gamma_{n_1}\delta_{n_1}+\delta_{n_1}^2-1)\end{aligned}$ &$\begin{aligned}(&\gamma_{n_1}-\gamma_{n_2},\ \delta_{n_1}-\delta_{n_2},\\ &\gamma_{m_1},\ \gamma_{m_2},\ \delta_{m_1}-1,\ \delta_{m_2}+1,\\ &x\gamma_{n_1}+2\delta_{n_1},\\ &\gamma_{n_1}^2+x\gamma_{n_1}\delta_{n_1}+\delta_{n_1}^2-1)\end{aligned}$
\end{tabular}

We can visualize all this information using the following undirected graph, with one node per component and an edge between two nodes indicating non trivial intersection between the corresponding components:

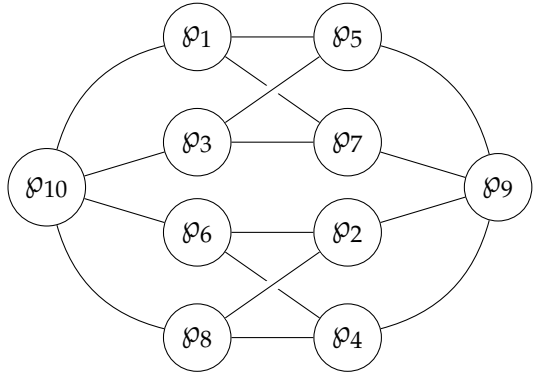
\begin{center}
\begin{tikzpicture}
\node[draw, circle, minimum size=15pt] (p10) at (0,0) {$\wp_{10}$};
\node[draw, circle, minimum size=15pt] (p1) at (2,2) {$\wp_1$};
\node[draw, circle, minimum size=15pt] (p3) at (2,.6) {$\wp_3$};
\node[draw, circle, minimum size=15pt] (p6) at (2,-.6) {$\wp_6$};
\node[draw, circle, minimum size=15pt] (p8) at (2,-2) {$\wp_8$};
\node[draw, circle, minimum size=15pt] (p5) at (4,2) {$\wp_5$};
\node[draw, circle, minimum size=15pt] (p7) at (4,.6) {$\wp_7$};
\node[draw, circle, minimum size=15pt] (p2) at (4,-.6) {$\wp_2$};
\node[draw, circle, minimum size=15pt] (p4) at (4,-2)  {$\wp_4$};
\node[draw, circle, minimum size=15pt] (p9) at (6,0) {$\wp_9$};
\draw (p10) to [bend left] (p1);
\draw (p10) to (p3);
\draw (p10) to (p6);
\draw (p10) to [bend right] (p8);
\draw (p1) to (p5);
\draw (p1) to (p7);
\draw [line width=5pt, white] (p3) to (p5);
\draw (p3) to (p5);
\draw (p3) to (p7);
\draw (p6) to (p2);
\draw (p6) to (p4);
\draw [line width=5pt, white] (p8) to (p2);
\draw (p8) to (p2);
\draw (p8) to (p4);
\draw (p5) to [bend left](p9);
\draw (p7) to (p9);
\draw (p2) to (p9);
\draw (p4) to [bend right] (p9);
\end{tikzpicture}
\end{center}

\subsection{Components as spaces of matrices}

As alluded to before, for values of $(n_1,m_1,n_2,m_2)$ there is a $TH[F_2]$-algebra homomorphism $\phi : R \to TH[F_2] = \Z[x,y,z]$
defined by
\[
\begin{matrix}
\gamma_{n_i}\mapsto c_{n_i}(x), &\delta_{n_i}\mapsto d_{n_i}(x), &\gamma_{m_i}\mapsto c_{m_i}(y), &\delta_{m_i}\mapsto d_{m_i}(y)\end{matrix} 
\] 
for $i=1,2$.
By equation~(\ref{eq:determinant_condition}), this homomorphism factors through the map $R\to R/J$ where 
\[\begin{aligned}J:=(&\gamma_{n_i}^2+\delta_{n_i}^2+x\gamma_{n_i}\delta_{n_i}-1\text{ for }i=1,2,\\ &\gamma_{m_i}^2+\delta_{m_i}^2+y\gamma_{m_i}\delta_{m_i}-1\text{ for }i=1,2)\subset I.\end{aligned}\]
Denote by $I^{\tt e}$ the ideal generated by the image of $\phi$. Observe that 
\[\begin{aligned}I^{\tt e} = (&zc_{n_1}(x)c_{m_1}(y)+xc_{n_1}(x)d_{m_1}(y)+yd_{n_1}(x)c_{m_1}(y)\\ &\quad+d_{n_1}(x)d_{m_1}(y)-d_{n_2}(x)d_{m_2}(y),\\ &c_{n_1}(x)d_{m_1}(y)+c_{n_2}(x)d_{m_2}(y),\\ &d_{n_1}(x)c_{m_1}(y)+d_{n_2}(x)c_{m_2}(y),\\ &c_{n_1}(x)c_{m_1}(y)+c_{n_2}(x)c_{m_2}(y))
\end{aligned}\]
equals the first ideal of Proposition~\ref{prop:explicit_equations}.
Denote by $f$ the canonical $\Z$-algebra homomorphism  
\begin{align*}
R/I \to & \Z[x,y,z]/I^{\tt e}
\end{align*} which is surjective since $R\to \Z[x,y,z]$ is. Let 
\[ {}^{\tt a}f : M := \Spec{\Z[x,y,z]/I^{\tt e}} \embed \Spec{R/I} \]
be the associated closed embedding of schemes. This fits into a diagram
\begin{center}
    \begin{tikzcd}
        M \arrow[r, hook] & \Spec{R/I} \arrow[r] & \Spec{R} \arrow[d]
        \\
        M^{\genirr}(\pi, \SL_2) \arrow[u, hook] \arrow[r] & M(\pi, \SL_2) \arrow[r] \arrow[ru, hook] \arrow[u, dashed] & M(F_2, \SL_2)
    \end{tikzcd}
\end{center}
Where $M^{\genirr}(\pi, \SL_2)$ lifts to $\Spec{R/I}$ by Proposition~\ref{prop:explicit_equations}. However, its image may be strictly smaller than $M$ if the elimination ideal
\[ (I^{\tt e} + (w \cdot \red - 1)) \cap \Z[x,y,z] \]
is larger than $I^{\tt e}$. Notice, as we remarked above there may not be a dashed map if some reducible representations do not satisfy the universal relations.
\par 
We found earlier an irreducible decomposition
\[ \Spec{R/I} = V(\wp_1) \cup \cdots \cup V(\wp_r) \]
defining a universal decomposition of $M$ into closed subschemes: 
\[ M_i := {}^{\tt a}f^{-1}(V(\wp_i)) = M \cap V(\wp_i) = V(\wp_i^{\tt e}). \]
Likewise, intersecting with the subscheme $M^{\genirr}(\pi, \SL_2) \subset M_i$ gives a decomposition into closed subschemes
\[ M_i^\circ \coloneq M_i \cap M^{\genirr}(\pi, \SL_2) \]
Let $\rho : \pi \to \mathrm{SL}_2(K)$ be any representation. Write $A:=\rho(a), B:=\rho(b)$ and recall that 
\[ c_{-r}(t)=-c_{r}(t) \quad \text{and} \quad d_{-r}(t)=c_{r+1}(t)=tc_{r}(t)+d_{r}(t). \] 

\begin{rmk}
    It is important to disambiguate between two very similar conditions: 
    \begin{enumerate}
        \item $A^n = s I$ where $s \in \{ \pm 1 \}$
        \item $c_n(\tr{A}) = 0$ and $d_n(\tr{A}) = s$
    \end{enumerate}
    First note that the determinantal condition ensures that $c_n(\tr{A}) = 0$ forces $d_n(\tr{A})^2 = 1$. Hence if $c_n(\tr{A}) = 0$ then $A^n = d_n(\tr{A}) I$ and $d_n(\tr{A}) = \pm 1$. However, the opposite implication does not quite hold. Indeed, if $A^n = s I$ then $c_n(\tr{A}) = 0$ if and only if either $\tr{A} \neq \pm 2$ or $n = 0$. Of course, if $c_n(\tr{A}) = 0$ then $d_n(\tr{A}) = s$. However, if $\tr{A} = \pm 2$ then $c_n(\tr{A}) = (\pm 1)^{n-1} n$ and $d_n(\tr{A}) = - (\pm 1)^n (n-1)$.
    \par 
    To summarize: (1) and (2) are equivalent if we modify (1) to include $A \neq \pm I$ unless $n = 0$.
\end{rmk}

Given this discussion, we introduce a new piece of notation. Write $A^n =_{nt} s I$ where $n$ is an integer and $s \in \{ \pm 1 \}$ to mean that $A^n = s I$ is a \textit{nontrivial solution} menaing that, if $n \neq 0$, then $A \neq \pm I$. When $n = 0$ this condition is vacuous. The previous discussion shows that $A^n =_{nt} s I$ is equivalent to $c_n(\tr{A}) = 0$ and $d_n(\tr{A}) = s$.

It follows that the pair of matrices $A,B$ has to satisfy at least one of the following set of conditions:

For $\rho \in M_1:$ 
\[A^{n_1}=_{nt}-I,\ A^{n_2}=_{nt}-I,\ B^{m_1+m_2}=_{nt} I \]

For $\rho \in M_2:$
\[A^{n_1}=_{nt}-I,\ A^{n_2}=_{nt}I,\ B^{m_1+m_2}=_{nt}-I.\]

For $\rho \in M_3:$
\[A^{n_1}=_{nt}I,\ A^{n_2}=_{nt}I,\ B^{m_1+m_2}=_{nt}I.\]

For $\rho \in M_4:$
\[A^{n_1}=_{nt}I,\ A^{n_2}=_{nt}-I,\ B^{m_1+m_2}=_{nt}-I.\]

For $\rho \in M_5:$
\[A^{n_1+n_2}=_{nt}I,\ B^{m_1}=_{nt}-I,\ B^{m_2}=_{nt}-I.\]

For $\rho \in M_6:$
\[A^{n_1+n_2}=_{nt}-I,\ B^{m_1}=_{nt}-I,\ B^{m_2}=_{nt}I.\]

For $\rho \in M_7:$
\[A^{n_1+n_2}=_{nt}I,\ B^{m_1}=_{nt}I,\ B^{m_2}=_{nt}I.\]

For $\rho \in M_8:$
\[A^{n_1+n_2}=_{nt}-I,\ B^{m_1}=_{nt}I,\ B^{m_2}=_{nt}-I.\]

For $\rho \in M_9:$
\[A^{n_1-n_2}=_{nt} -I,\ B^{m_1 - m_2}=_{nt} I,\ \mathrm{tr}(A^{n_1}B^{m_1}) = 0.\]

For $\rho \in M_{10}:$
\[A^{n_1-n_2}=_{nt} I,\ B^{m_1-m_2}=_{nt} -I,\ \mathrm{tr}(A^{n_1}B^{m_1}) = 0.\]

For $M_1$ we illustrate why these conditions are equivalent. The first four generators involve only $x$ only are already in the desired form: $A^{n_i} =_{nt} - I$. To show that $c_{m_1 + m_2}(y) \in \wp_1^{\tt e}$ we expand
\begin{align*}
    c_{m_1 + m_2}(y) &= y c_{m_1}(y) c_{m_2}(y) + c_{m_1}(y) d_{m_2}(y) + c_{m_2}(y) d_{m_1}(y)
    \\
    & = c_{m_2}(y) (y c_{m_1}(y) + d_{m_1}(y) - d_{m_2}(y)) 
    \\ & \quad \quad  + (c_{m_1}(y) + c_{m_2}(y)) d_{m_2}(y) \in \wp_1^{\tt e}
\end{align*}
To show that $d_{m_1 + m_2}(y) - 1 \in \wp_1^{\tt e}$ we expand
\begin{align*}
    d_{m_1 + m_2}(y) - 1 
    & = d_{m_1}(y) (d_{m_2}(y) - y c_{m_1}(y) - d_{m_1}(y))  
    \\
    & \quad+ (c_{m_1}(y)^2 + y c_{m_1}(y) d_{m_1}(y) + d_{m_1}(y)^2 - 1) 
    \\
    &\quad  - (c_{m_1}(y) + c_{m_2}(y)) c_{m_2}(y) \in \wp_1^{\tt e}
\end{align*}
Furthermore it is clear that $B^{m_1 + m_2} = 1$ implies the vanishing of the generators of $\wp_1^{\tt e}$. 
The others for $i = 1, \dots, 8$ are similar. For $M_9$ note that 
\begin{align*}
    & \tr{ A^{n_1} B^{m_1}}  = \tr{[(c_{n_1}(x) A + d_{n_1}(x) I)(c_{m_1}(y) B + d_{m_1}(y) I)]}
    \\
    & = z c_{n_1}(x) c_{m_1}(y) + x c_{n_1}(x) d_{m_1}(y) + y d_{n_1}(x) c_{m_1}(y) + 2 d_{n_1}(x) d_{m_1}(y)
\end{align*}
which is the additional generator appearing in $\wp^{\tt e}$.
The discussion for $M_{10}$ is similar. 

\subsection{Dimensions of components} \label{section:dimensions}

\begin{defn}
Let $n_1,n_2$ be nonzero integers and $s_1,s_2\in \{\pm1\}$ signs. Write $g:=\mathrm{gcd}(n_1,n_2)$. We denote by $C(n_1,n_2,s_1,s_2)$ the condition 
\[\begin{aligned}g>1\text{ and }s_i=&(-1)^{n_i/g}\text{ for }i=1,2\\ &\text{or}\\ g>2\text{ and }s_i=&1\text{ for }i=1,2.\end{aligned}\]
\end{defn}

\begin{rmk}
Let $\nu_2$ be the $2$-adic valuation of $\Q$. Condition $C(n_1,n_2,s_1,s_2)$ can be restated for any particular choice of signs $s_i\in\{\pm1\}$ as follows:
\begin{enumerate}
\item $C(n_1,n_2,1,1)$ is exactly $g>2$,
\item $C(n_1,n_2,-1,1)$ is the condition $g>1$, $-1=(-1)^{n_1/g}$ and $1=(-1)^{n_2/g}$. The sign conditions are equivalent to $\nu_2(n_1/g)=0$ and $\nu_2(n_2/g)>0$. Moreover, these last two inequalities are equivalent to the inequality $\nu_2(n_1)<\nu_2(n_2)$,
\item $C(n_1,n_2,1,-1)$ is equivalent to $g>1$ and $\nu_2(n_1)>\nu_2(n_2)$,
\item $C(n_1,n_2,-1,-1)$ means, by definition, $g>1$ and $-1=(-1)^{n_i/g}$ for $i=1,2$. The conditions on signs are equivalent to the equality $\nu_2(n_1)=\nu_2(n_2)$. So $C(n_1,n_2,-1,-1)$ can be restated as $g>1$ and $\nu_2(n_1)=\nu_2(n_2)$.
\end{enumerate}
\end{rmk}

The importance of this condition is the following lemma

\begin{lemma} 
\label{lemma:system_coprime}
Let $k$ be an algebraically closed field of characteristic zero. 
Let $n_1, n_2$ be nonzero integers and $s_1,s_2\in\{\pm1\}$ signs. The system \[\left\{\begin{matrix}c_{n_1}(t) =0, & d_{n_1}(t)=s_1,
\\ c_{n_2}(t)=0, & d_{n_2}(t) =s_2;
\end{matrix}\right.\] 
has a solution in $k$ if and only if $C(n_1,n_2,s_1,s_2)$ holds.
\end{lemma}

By applying Lemma~\ref{lemma:system_coprime} we obtain the following computation of the dimensions of the ``middle layers'' of the intersection diagram. To shorten the notation, we will use the somewhat nonstandard practice of writing
\[ f(x) = 
\begin{cases}
a & P
\\
b & \text{else}
\\
c & Q
\end{cases} \]
to mean: if $P$ is true the $f(x) = a$, otherwise if $Q$ is true then $f(x) = c$, and otherwise if both $P$ and $Q$ are false then $f(x) = b$.

\begin{prop} \label{prop:mid_component_dimension}
\begin{align*}
\dim V(\wp_1^{\tt e})_{\Q} &= \begin{cases} -\infty & \neg C(n_1, n_2,-1,-1) \text{ or } m_1 + m_2 = \pm 1, \pm 2 \\ 1 & \text{else} \\ 2 & C(n_1, n_2,-1,-1) \text{ and } m_1 + m_2 = 0 
\end{cases}
\\
\dim V(\wp_2^{\tt e})_{\Q} &= \begin{cases} -\infty & \neg C(n_1, n_2, -1, 1) \text{ or } m_1 + m_2 = 0, \pm 1 \\ 1 & C(n_1, n_2, -1, 1) \text{ and } m_1 + m_2 \neq 0, \pm 1
\end{cases}
\\
\dim V(\wp_3^{\tt e})_{\Q} &= \begin{cases} -\infty & \neg C(n_1, n_2, 1, 1) \text{ or } m_1 + m_2 = \pm 1, \pm 2 \\ 1 & \text{else} \\ 2 & C(n_1, n_2, 1, 1) \text{ and } m_1 + m_2 = 0
\end{cases}
\\
\dim V(\wp_4^{\tt e})_{\Q} &= \begin{cases} -\infty & \neg C(n_1, n_2, 1, -1) \text{ or } m_1 + m_2 = 0, \pm 1 \\ 1 & C(n_1, n_2, 1, -1) \text{ and } m_1 + m_2 \neq 0, \pm 1
\end{cases}
\\
\dim V(\wp_5^{\tt e})_{\Q} &= \begin{cases} -\infty &  \neg C(m_1, m_2, -1, -1) \text{ or } n_1 + n_2 = \pm 1, \pm 2 \\ 1 & \text{else} \\ 2 & C(m_1, m_2,-1,-1) \text{ and } n_1 + n_2 = 0
\end{cases}
\\
\dim V(\wp_6^{\tt e})_{\Q} &= \begin{cases} -\infty & \neg C(m_1, m_2, -1, 1) \text{ or } n_1 + n_2 = 0, \pm 1 \\ 1 & C(m_1, m_2, -1, 1) \text{ and } n_1 + n_2 \neq 0, \pm 1
\end{cases}
\\
\dim V(\wp_7^{\tt e})_{\Q} &= \begin{cases} -\infty & \neg C(m_1, m_2, 1, 1) \text{ or } n_1 + n_2 = \pm 1, \pm 2 \\ 1 & \text{else} \\ 2 & C(m_1, m_2, 1, 1) \text{ and } n_1 + n_2 = 0
\end{cases}
\\
\dim V(\wp_8^{\tt e})_{\Q} &= \begin{cases} -\infty & \neg C(m_1, m_2, 1, -1) \text{ or } n_1 + n_2 = 0, \pm 1 \\ 1 & C(m_1, m_2, 1, -1) \text{ and } n_1 + n_2 \neq 0, \pm 1
\end{cases}
\end{align*}
\end{prop}

\begin{proof}
Recall that $c_{\pm 1}(x) = 0$ has no solutions and $c_{\pm 2}(x) = 0$ and $d_{\pm 2}(x) = 1$ has no simultaneous solutions. These, along with the lemma give the cases where $V(\wp_i^{\tt e})_{\Q}$ is empty. 
\par 
When $n = 0$ the system $A^n =_{nt} I$ (i.e.\ $c_n(x) = 0$ and $d_n(x) = 1$) imposes no condition and $A^n =_{nt} -I$ is unsatisfiable giving an additional possibility for when the component is empty. 
\par 
Since there is no condition fixing $z$, whenever the component is nonempty it has dimension at least $1$. Without loss of generality, we focus on $M_1$ and $M_2$. The conditions on $A$ are always nontrivial (since we assume $n_1, n_2 \neq 0$) and fix $\tr{A}$ onto some zero-dimensional subscheme. Therefore the dimension is determined by whether or not the condition on $B$, namely $B^{m_1 + m_2} =_{nt} \pm 1$ imposes a condition on $\tr{B}$. It does if and only if $m_1 + m_2 \neq 0$.
\end{proof}

\begin{proof}[Proof of Lemma~\ref{lemma:system_coprime}]
See \S \ref{section:polynomials} for properties of the polynomials $c_n, d_n$. If the system has a solution, then $c_{n_1}, c_{n_2}$ have a common root so, setting $g = (n_1, n_2)$, $c_{g}$ must have a root. Indeed, by Lemma~\ref{lemma:gcd} we know $c_g(t)=\mathrm{gcd}(c_{n_1}(t),c_{n_2}(t))$. Hence $g > 1$. Let $M$ be a matrix whose trace is a solution. Then $M^{n_i} = s_i I$ so $M^g = M^{a n_1 + b n_2} = s_1^a s_2^b I$. But then $M^{n_i} = (M^g)^{n_i/g} = (s_1^a s_2^b)^{n_i / g} I$ so $s_i = (s_1^a s_2^b)^{n_i / g}$. If $s_1^a s_2^b = -1$ then we recover the condition $g > 1$ and $s_i = (-1)^{n_i/g}$. If $s_1^a s_2^b = 1$ then we have $M^g = I$ and moreover $c_g(\tr{M}) = 0$ which implies $\tr{M} \neq \pm 2$ but if $g = 2$ then $M^2 = I$ implies $\tr{M} = \pm 2$. Hence $g > 2$ and we reach the second condition: $g > 2$ and $s_i = 1$. 
\par 
Conversely, consider the matrix 
\[ M_g = \begin{pmatrix}
\zeta_{2g} & 0
\\
0 & \zeta_{2g}^{-1}
\end{pmatrix} \]
which satisfies $M_g^{g} = -I$ so $M_g^{n_i} = (M_g^g)^{n_i/g} = (-1)^{n_i/g}I = s_i I$. This satisfies the property in the first case. For $g \neq 2$ and $s_1 = s_2 = 1$ we can set $M = M_g^2$ then $M^g = I$ so $M^{n_i} = 2$ and $M \neq \pm I$ so $\tr{M}$ satisfies the system.
\end{proof}

\begin{lemma} \label{prop:mid_components_integrality}
For $i = 1, \dots, 8$, we have $M_i \subset M^{\genirr}(\pi, \SL_2)$ and as long as $M_i \neq \emptyset$ then $M_i(\Z[\mu_\infty]) \neq \emptyset$.
\end{lemma}

\begin{proof}
Recall that $M^{\genirr}(\pi, \SL_2)$ is the closure of $M \cap D(\red)$ in $M$ so it suffices to show that $\red$ does not vanish at each generic point. 
Assume $M_i \neq \emptyset$. Notice that no equation in $\wp_i$ involves $z$. Hence, as long as there is a solution for some $x,y \in K$ then $z$ is completely unconstrained. As long as $K$ is infinite there will be a choice of $z$ such that $\red(x,y,z) \neq 0$ since it is a monic quadratic over $\Z[x,y]$. Hence, such a representation defines a point of $M^{\genirr}(\pi, \SL_2)$. This discussion holds whenever $\Spec{K} \to M_i$ is a generic point since $M_i \cong \Spec{\Z[x,y]/(\Z[x,y] \cap \wp^{\tt e})} \times \Spec{\Z[z]}$ so $k(z) \subset K$ and hence $\red \neq 0$ in $K$. Furthermore, the equations fixing $x,y$ are monic and moreover their roots all lie in cyclotomic fields (since they are all of the form $c_k = 0$ or $d_k = \pm 1$) and $k$ so any point with $z \in \Z[\mu_\infty]$ is a $\Z[\mu_\infty]$-point. 
\end{proof}

Similarly, we obtain a description of components $9,10$.

\begin{prop}
\begin{align*}
\dim V(\wp_9^{\tt e})_{\Q} &= 
\begin{cases}
-\infty & n_1 - n_2 = 0, \pm 1 \text{ or } m_1 - m_2 = \pm 1, \pm 2 
\\
& \text{ or } 
\\
& \Big\{ n_1 - n_2 \divides n_1 \text{ and } \neg C(m_1+m_2, m_1 - m_2, -1,1) \Big\} 
\\
& \text{ or } 
\\
& \Big\{ m_1 - m_2 \divides 2m_1 \text{ and } \neg C(n_1+n_2, n_1 - n_2, 1,-1) \Big\}
\\
0 & \text{else}
\\
1 & m_1 = m_2 \text{ or } 
\\
& \Big\{ [C(n_1, n_2, -1,1) \text{ or } C(n_1, n_2, 1,-1)] 
\\
& \quad \text{ and } C(m_1 + m_2, m_1 - m_2, -1, 1) \Big\}
\\
& \text{ or }
\\
& \Big\{ [C(m_1, m_2, 1,1) \text{ or } C(m_1, m_2, -1,-1)] 
\\
& \quad \text{ and } C(n_1 + n_2, n_1 - n_2, 1, -1) \Big\}
\end{cases}
\\
\dim V(\wp_{10}^{\tt e})_{\Q} &= 
\begin{cases}
-\infty & m_1 - m_2 = 0, \pm 1 \text{ or } n_1 - n_2 = \pm 1, \pm 2 
\\
& \text{ or } 
\\
& \Big\{ m_1 - m_2 \divides m_1 \text{ and } \neg C(n_1+n_2, n_1 - n_2, -1,1) \Big\} 
\\
& \text{ or } 
\\
& \Big\{ n_1 - n_2 \divides 2n_1 \text{ and } \neg C(m_1+m_2, m_1 - m_2, 1,-1) \Big\}
\\
0 & \text{else}
\\
1 & n_1 = n_2 \text{ or } 
\\
& \Big\{ [C(m_1, m_2, -1,1) \text{ or } C(m_1, m_2, 1,-1)] 
\\
& \quad \text{ and } C(n_1 + n_2, n_1 - n_2, -1, 1) \Big\}
\\
& \text{ or }
\\
& \Big\{ [C(n_1, n_2, 1,1) \text{ or } C(n_1, n_2, -1,-1)] 
\\
& \quad \text{ and } C(m_1 + m_2, m_1 - m_2, 1, -1) \Big\}
\end{cases}
\end{align*}
\end{prop}

\begin{proof}
In components $i = 9,10$ observe that there is only one equation involving $z$ and it appears exactly once with coefficient $c_{n_1}(x) c_{m_1}(y)$. Hence $z$ is uniquely determined by $x,y$ unless $A^{n_1} =_{nt} \pm I$ or $B^{m_1} =_{nt} \pm I$ in which case $z$ is completely free. This will be the main observation in the argument. 

Since the equations defining $M_9$ and $M_{10}$ are identical up to swapping $A,B$ and $x,y$ and $n_i, m_i$, without loss of generality, we may focus on $M_9$. Recall $M_9$ is defined by the following relations in terms of matrices
\[ A^{n_1 - n_2} =_{nt} - I \quad B^{m_1 - m_2} =_{nt} I \quad \tr{A^{n_1} B^{m_1}} = 0. \]
We first compute when $M_9$ is empty. Suppose $n_1 - n_2 = \pm 1$ then the condition $A^{n_1 - n_2} =_{nt} -I$ is contradictory so indeed $M_9 = \emptyset$. If $m_1 - m_2 = \pm 1, \pm 2$ then the condition $B^{m_1 - m_2} =_{nt} I$ is contradictory so indeed $M_9 = \emptyset$.
\par 
Suppose neither of these hold, then the first two conditions are satisfiable so we need to see when the third is as well. Notice the third equation is equivalent to the trace equation
\[ zc_{n_1}(x)c_{m_1}(y)+xc_{n_1}(x)d_{m_1}(y)+yc_{m_1}(y)+d_{n_1}(x)+2d_{n_1}(x)d_{m_1}(y)=0, \]
so if $c_{n_1}(x) c_{m_1}(y) \neq 0$ this is always satisfiable for some $z$. Hence we only need to consider the cases that $A^{n_1} =_{nt} \pm 1$ or $B^{m_1} =_{nt} \pm I$. 
\par 
Suppose $A^{n_1} =_{nt} s I$ then $A^{n_2} = -s I$ from the first equation. From the word we see that
\[ A^{n_1} B^{m_1} A^{n_2} B^{m_2} = -B^{m_1 + m_2} = I \]
but we also have $B^{m_1 - m_2} =_{nt} I$ and hence to have a solution $C(m_1 + m_2, m_1 - m_2, -1, 1)$ must hold. The only additional constraint on $A$ is $A^{n_1 - n_2} =_{nt} -I$ so $A^{n_1} =_{nt} s I$ is only forced if $n_1 - n_2 \divides n_1$ (equivalently if $n_1 - n_2 \divides n_2$) because otherwise $A^{n_1 - n_2} =_{nt} - I$ admits solutions that do not satisfy $A^{n_1} = \pm I$. Similarly, if $B^{m_1} = \pm I$ the same argument shows that $C(n_1 + n_2, n_1 - n_2, 1, -1)$ must be satisfied. But $B^{m_1 - m_2} =_{nt} I$ forces $B^{m_1} = \pm I$ exactly when $m_1 - m_2 \divides 2 m_1$. Hence if the conditions listed under $-\infty$ in the theorem statement hold, then the component must be empty. 
\par 
Otherwise, there exists a solution but notice that $A^{n_1 - n_2} =_{nt} - I$ always is either nontrivial (restricting $x = \tr{A}$ to a zero-dimensional subscheme) or unsatisfiable. The same is true of $B^{m_1 - m_2} =_{nt} I$ except when $m_1 = m_2$ then this imposes no condition.  Hence if $m_1 \neq m_2$ then the condition on $B$ restricts $\tr{B}$ to a zero-dimensional subscheme. 
\\
Suppose $m_1 = m_2$, then the component will be at least $1$-dimensional. Indeed we can rewrite the last generator of the ideal as giving the equation
\[ z c_{n_1}(x) c_{m_1}(y) + c_{n_1+1}(x) d_{m_1}(y) + c_{m_1 + 1}(y) d_{n_1}(x) = 0\]
since $c_{n_1}(x)$ and $c_{n_1+1}(x)$ cannot vanish simultaneously we see that this equation is always nontrivial in $y,z$ meaning the dimension is at most $1$. Hence we conclude.
\\
From now on, suppose $m_1 \neq m_2$ hence the condition $B^{m_1 - m_2} =_{nt} I$ constrains $\tr{B}$ to a zero-dimensional subscheme and $B \neq \pm I$. As long as $A^{n_1} =_{nt} s_n I$ and $B^{m_1} =_{nt} s_m I$ are not satisfied for any choice of signs $s_n, s_m$ then $z$ is uniquely determined by the remaining equation so the component will be zero dimensional. Hence, to understand the $1$-dimensional structure, we want to know when $A^{n_1} =_{nt} s_1 I$ or $B^{m_1} =_{nt} s_2 I$ is \textit{consistent} with the defining equation (when we analyzed emptiness we needed to see when this condition was forced). The system $A^{n_1 -  n_2} =_{nt} -I$ and $A^{n_1} =_{nt} \pm I$ has a solution if and only if $C(n_1, n_2, \pm 1, \mp 1)$. When this is satisfied, in order to satisfy the word (or equivalently the trace condition $\tr{A^{n_1} B^{m_1}} = 0$) we must have
\[ B^{m_1 + m_2} =_{nt} - I \]
but also $B^{m_1 - m_2} =_{nt} I$ so $C(m_1 + m_2, m_1 - m_2, -1, 1)$ must be satisfied. Similarly, the system $B^{m_1 - m_2} =_{nt} I$ and $B^{m_1} =_{nt} \pm I$ has a solution if and only if $C(m_1, m_2, \pm 1, \pm 1)$ and will imply that $A^{n_1 + n_2} =_{nt} I$ must be solvable along with $A^{n_1 - n_2} =_{nt} -I$ hence $C(n_1 + n_2, n_1 - n_2, 1, -1)$ is satisfied. Therefore we have found exactly the conditions for $\dim{(M_9)_{\Q}} = -\infty, 1$ and shown it has dimension $\le 1$ so the remaning cases are exactly those where the dimension is zero. 
\end{proof}

In the proof we see explicitly that $1$-dimensional components of $M_9$ (resp.\ $M_{10}$) are either due to $\tr{B}$ (resp.\ $\tr{A}$) being unconstrained which occurs only if $m_1 = m_2$ (resp.\ $n_1 = n_2$) and otherwise they arise from $z$ unconstrained which only happens inside the locus $M_1 \cup \cdots \cup M_8$ by Remark~\ref{rmk:constraining_z}. For future convenience we say that $(\ast)_9$ is satisfied if $m_1 \neq m_2$ (resp. $(\ast)_{10}$ if $n_1 \neq n_2$). 

\begin{prop} \label{prop:end_components}
Let $i = 9,10$ and an irreducible component $Z \subset M_i$,
\begin{enumerate}
    \item if $(\ast)_i$ and $\dim{Z_{\Q}} > 0$ then $Z \subset M^{\genirr}(\pi, \SL_2)$ and $Z(\ZZ[\mu_\infty]) \neq \emptyset$
    \item if $\dim{Z_{\Q}} = 0$ and $Z(\ol{\ZZ}) = \emptyset$ then $Z \subset M^{\genirr}(\pi, \SL_2)$ and $Z_{\Q} \subset M^{\irr}(\pi, \SL_2)_{\Q}$.
\end{enumerate}
\end{prop}

\begin{proof}
Suppose $\dim{Z_{\Q}} > 0$. From our description of $M_i$ and the remark above, $z$ must be unconstrained and by Remark~\ref{rmk:constraining_z} this implies $Z \subset M_1 \cup \cdots \cup M_8 \subset M^{\genirr}(\pi, \SL_2)$ and $Z(\Z[\mu_\infty]) \neq \emptyset$ by Lemma~\ref{prop:mid_components_integrality}. 
\par 
Otherwise, suppose $\dim{Z_{\Q}} = 0$ then $x,y$ are constrained by monic equations so $Z(\ol{\Z}) = \emptyset$ exactly if $z \notin \ol{\Z}$ for any point $(x,y,z) \in Z(\ol{\Q})$. However, in this case $\red(x,y,z) \neq 0$ since $\red$ is a monic quadratic in $z$ over $\Z[x,y]$. Hence if $x,y \in \ol{\Z}$ and $\red(x,y,z) = 0$ then $z \in \ol{\Z}$. Therefore $Z \subset M^{\genirr}(\pi, \SL_2)$ and $Z_{\Q} \subset M^{\irr}(\pi, \SL_2)_{\Q}$.
\end{proof}

These propositions have proven exactly the content of Theorem~\ref{thm:components}.

\begin{example}
It is possible for $M_i^\circ \subset M_i$ to be a strict inclusion. Note that because $M_i^\circ$ is the closure of $M_i \cap D(\red)$ in $M_i$ they can only differ by deleting irreducible components. Indeed for $\pi = \Gamma_2$, it is easy to check that $M^{\genirr}(\Gamma_2, \SL_2) = M_9^\circ = M_{9}$ is zero dimensional, $M_{10}^{\circ} = \emptyset$,  but $M_{10}$ is $1$-dimensional. 
\end{example}

\begin{prop}
Each $(M_i)_{\Q}$, and hence $M^{\genirr}(\pi, \SL_2)_{\Q}$, is reduced. 
\end{prop}

\begin{proof}
    This follows from the explicit descriptions above and the separability of the polynomials $c_n(t)$. 
\end{proof}

\begin{proof}[Proof of Theorem~\ref{thm:integrality}]
Let $m_1 + m_2 = \pm 1$ or $n_1 + n_2 = \pm 1$. It is easy to see from our dimension formulas this implies the required emptiness and dimension results. Focusing on $M_9$ notice that it is defined by
\begin{align*} 
\wp^{\tt e}_9 = (&c_{n_1 - n_2}(x), c_{m_1 - m_2}(y), d_{n_1 - n_2}(x) + 1, d_{m_1 - m_2}(y) -1,
\\
& z c_{n_1}(x) c_{m_1}(y) + c_{n_1+1}(x) d_{m_1}(y) + c_{m_1 + 1}(y) d_{n_1}(x))
\end{align*}
Recall that the determinant conditions evaluate to zero identically when the $c,d$ polynomials are substituted. 
This fixes $x = \zeta_{2s}$ and $y = \zeta_t$ where $s = n_1 - n_2$ and $t = m_1 - m_2$. Therefore, plugging in the formulas for $c,d$ of diagonal matrices in terms of eigenvalues, the last equation fixes $z$ to the value in the theorem statement. We must only worry that this point does not land in $M_9^\circ$. Indeed this can happen by the above example. However, if $z$ is not an algebraic integer then $\red(x,y,z) \neq 0$ since $x,y \in \ol{\Z}$ and $\red \in k[x,y][z]$ is monic of degree $2$ in $z$. Hence indeed this point lies in $M_9^\circ$. Since this describes all possible $\ol{\Q}$-points and $\dim{(M_9^\circ)_{\Q}} = 0$ the characteristic zero fiber is covered by the union of these inclusions. The integrality properties then follow by definition. 
\end{proof}

\bibliographystyle{amsalpha}
\bibliography{refs.bib}

\end{document}